\documentclass[final,hidelinks,onefignum,onetabnum]{siamart250211}
\usepackage[utf8]{inputenc}
\usepackage{amsmath,amsfonts,mathbbol,bm,soul}
\usepackage{booktabs, multirow}
\usepackage{mathtools}
\usepackage{graphicx}
\usepackage{tabulary}
\usepackage{subfig}
\usepackage{hyperref}
\usepackage{soul}
\usepackage{colortbl}
\usepackage{amssymb}
\usepackage{changepage,threeparttable}
\usepackage{mathabx,epsfig}
\usepackage{color, listings}
\usepackage[normalem]{ulem}
\usepackage{cancel}
\DeclareSymbolFontAlphabet{\mathbb}{AMSb}
\DeclareSymbolFontAlphabet{\mathbbm}{bbold}

\newcommand{\vecf}[1]{\bm{#1}}

\newcommand{\fnt}[1]{\bm{\mathsf{#1}}}

\DeclareRobustCommand{\revdel}[1]{%
  \ifmmode
    {\color{red}\cancel{#1}}%
  \else
    \textcolor{red}{\sout{#1}}%
  \fi
}

\usepackage{amsfonts,amssymb}

\renewcommand{\aligned}[1]{&#1}
\newcommand{\labell}[1]{\addtocounter{equation}{1}\tag{\theequation}\label{#1}}

\definecolor{dkgreen}{rgb}{0,0.6,0}
\definecolor{gray}{rgb}{0.5,0.5,0.5}
\definecolor{mauve}{rgb}{0.58,0,0.82}

\lstnewenvironment{python}[1][]{%
  \lstset{language=Python, #1}%
}{}

\lstnewenvironment{java}[1][]{%
  \lstset{language=Java, #1}%
}{}

\lstnewenvironment{javascript}[1][]{%
  \lstset{language=Java, #1} 
}{}

\usepackage{lipsum}
\usepackage{amsfonts}
\usepackage{graphicx}
\usepackage{epstopdf}
\usepackage{algorithmic}
\ifpdf
  \DeclareGraphicsExtensions{.eps,.pdf,.png,.jpg}
\else
  \DeclareGraphicsExtensions{.eps}
\fi

\newsiamremark{remark}{Remark}
\newsiamremark{hypothesis}{Hypothesis}
\crefname{hypothesis}{Hypothesis}{Hypotheses}
\newsiamthm{claim}{Claim}
\newsiamremark{fact}{Fact}
\crefname{fact}{Fact}{Facts}

\headers{ESFD methods via ECAV and KL}{B. Christner and J. Chan}

\title{Entropy stable finite difference methods via entropy correction artificial viscosity and knapsack limiting \thanks{
\funding{Supported by National Science Foundation awards DMS-1943186 and DMS-223148.}}}

\author{Brian Christner\thanks{Department of Computational Applied Mathematics and Operations Research, Rice University 
  (\email{bc71@rice.edu}).}
\and Jesse Chan\thanks{Oden Institute for Computational Engineering and Sciences and Department of Aerospace Engineering and Engineering Mechanics, The University of Texas at Austin}
  (\email{jesse.chan@oden.utexas.edu}).
  }

\usepackage{amsopn}

\ifpdf
\hypersetup{
  pdftitle={Entropy stable finite difference methods via entropy correction artificial viscosity and knapsack limiting},
  pdfauthor={B. Christner and J. Chan}
}
\fi

\begin{document}
\maketitle

\begin{abstract}
Entropy stable methods have become increasingly popular in the field of computational fluid dynamics. They often work by satisfying some form of a discrete entropy inequality: a discrete form of the 2nd law of thermodynamics. Schemes which satisfy a (semi) discrete entropy inequality typically behave much more robustly, and do so in a way that is hyperparameter free. Recently, a new strategy was introduced to construct entropy stable discontinuous Galerkin methods: knapsack limiting, which blends together a low order, positivity preserving, and entropy stable scheme with a high order accurate scheme, in order to produce a high order accurate, entropy stable, and positivity preserving scheme. Another recent strategy introduces an entropy correction artificial viscosity into a high order scheme, aiming to satisfy a cell entropy inequality. 

In this work, we introduce the techniques of knapsack limiting and artificial viscosity for finite difference discretizations. The proposed schemes preserve high order accuracy in sufficiently smooth conditions, are entropy stable, and are hyperparameter free. Moreover, the proposed knapsack limiting scheme provably preserves positivity for the compressible Euler equations. Both schemes achieve this goal without significant performance tradeoffs compared to state of the art stabilized schemes.
\end{abstract}



\section{Introduction}\label{Introduction}

In computational science, particularly when addressing multi-dimensional systems, numerically solving partial differential equations (PDEs) is critical. A classical method for numerically approximating the solutions to PDEs is the finite difference method. One of the primary advantages of finite difference methods are their ability to achieve arbitrarily high spatial accuracy by using wide stencils. Moreover, they are often the simplest and most computationally inexpensive schemes to implement. This makes them a valuable tool for capturing fine details of complex systems. The high order accuracy enables the numerical solution to approximate the true solution with a coarser mesh, reducing the computational time and resource expenditure. 

Despite its advantages, high order finite difference methods are not without limitations, the most notable being a lack of robustness. High order methods often exhibit robustness issues or may fail to converge to accurate solutions. In response, essentially non-oscillatory (ENO) and weighted ENO (WENO) schemes have been proposed to mitigate these issues. These schemes can be applied to finite difference to construct an arbitrarily high order method, whose resulting simulations avoid oscillations around shocks. These schemes can also be used to construct solutions which satisfy nodal bounds. For instance, they can be used to satisfy positivity of density and pressure for the compressible Euler equations. Though WENO schemes avoid oscillations in their simulations, they are generally not provably entropy stable \cite{Fisher13}. 

One alternate approach to constructing high order numerical methods which has gained traction within computational fluid dynamics (CFD) are discontinuous Galerkin (DG) methods. A common approach to resolve issues of robustness within DG methods is to satisfy a cell entropy inequality. The resulting methods are often called entropy stable DG (ESDG) methods. They often work by satisfying a (semi) discrete entropy inequality: a discrete form of the 2nd law of thermodynamics. When discretizing conservation laws, it is often the case that the 2nd law of thermodynamics is not satisfied under discretization. By enforcing such an inequality, the resulting solutions typically behave much more robustly, and do so in a way that is hyperparameter free. 

A recent technique for enforcing semi-discrete entropy stability for nodal DG methods, also known as DG spectral element methods (DGSEM), is \textit{knapsack limiting}. At each right-hand-side evaluation, knapsack limiting optimally blends together a low order, entropy stable, positivity preserving update, with the high order accurate update coming from DGSEM. The resulting update is entropy stable, arbitrarily high order accurate in smooth regions, and satisfies a relative positivity condition. Moreover, it was shown in theory \cite{Vilar25} and in practice \cite{Lin24, Christner25} that the resulting scheme preserves high order accuracy in sufficiently smooth regions. The methods developed in this paper are related to knapsack limiting, and utilize the same conceptual approach of computing optimal stabilization parameters using an optimization formulation with an inexpensive solver. 

Another recent strategy in the field of DG methods for satisfying a semi-discrete entropy inequality are entropy correction artificial viscosity methods \cite{JesseAV}. This scheme introduces an artificial viscosity into a high order accurate update, in order to satisfy a cell entropy inequality. The resulting methods are much more intuitive than the knapsack limiting scheme. This resulting scheme also exhibits arbitrarily high order accuracy for sufficiently regular cases. 

In this work, we introduce the techniques of artificial viscosity entropy correction from the field of DG methods into the realm of finite difference methods. The first proposed scheme, entropy correction artificial viscosity for finite differences (ECAV-FD) preserves high order accuracy in sufficiently smooth regions, is entropy stable, and is parameter free. The second proposed scheme, knapsack limiting for finite differences (KL-FD) is a generalization of the ECAV-FD scheme, which can also provably preserve positivity.

\section{High Order Finite Difference Methods}\label{High Order Finite Difference Methods}

We will now discuss the finite difference discretization, for which we will utilize summation by parts (SBP) finite difference operators. For the sake of brevity, we will only describe how to construct SBP operators in $d\in\left\{1,2\right\}$ dimensions, with the note that extension to higher dimensions is possible. The derived schemes will be written in such a way that an arbitrary number of dimensions are applicable. 

We consider a hyperbolic conservation law imposed over domain $\Omega\subseteq\mathbb{R}^{d}$ with the imposed fluxes $\vecf{f}_{k}:\mathbb{R}^{v}\to\mathbb{R}^{v}$:
\begin{align*}\frac{\partial\vecf{u}}{\partial t}+\sum_{k=1}^{d}\frac{\partial\vecf{f}_{k}\left(\vecf{u}\right)}{\partial x_{k}}=\fnt{0}\labell{eq:ConservationLaw},\end{align*}
where $\vecf{u}:\mathbb{R}^{d}\times\left[0,T\right]\to\mathbb{R}^{v}$, $v$ is the number of variables, and $T$ is the final time. The domain $\Omega$ is discretized using equally spaced points $\fnt{x}\in\Omega$. $n$ is the total number of nodes. The discretized solution can now be expressed $\fnt{u}:\left[0,T\right]\to\left(\mathbb{R}^{v}\right)^{n}$, such that $\fnt{u}\left(t\right)$ consists of pointwise evaluations of $\vecf{u}$ along the nodes $\fnt{x}$. Now, we consider $\fnt{D}_{1},...,\fnt{D}_{d}\in\mathbb{R}^{n\times n}$ finite difference matrices of accuracy order $N$ along each dimension. At this point, the conservation law can be discretized
\begin{align*}\fnt{M}\frac{\text{d}\fnt{u}^{H}}{\text{d}t}+\sum_{k=1}^{d}\left(\fnt{D}_{k}\vecf{f}_{k}\left(\fnt{u}\right)\right)=\fnt{0},\labell{eq:HOM}\end{align*}
where $\fnt{M}\in\mathbb{R}^{n\times n}$ is a diagonal mass matrix corresponding to $\fnt{D}_{1},...,\fnt{D}_{d}$. An equivalent \textit{flux-differencing} form can be derived:
\begin{align*}\sum_{k=1}^{d}\left(\fnt{D}_{k}\vecf{f}_{k}\left(\fnt{u}\right)\right)_{i}\aligned{=}\sum_{k=1}^{d}\sum_{j}^{ }2\fnt{D}_{k,ij}\left(\frac{\vecf{f}_{k}\left(\fnt{u}_{i}\right)+\vecf{f}_{k}\left(\fnt{u}_{j}\right)}{2}\right),\labell{eq:FluxDiffL1}\\
\aligned{=}\sum_{j}^{ }\left\lVert\fnt{n}_{ij}\right\rVert\cdot\left(\frac{\vecf{f}\left(\fnt{u}_{i}\right)+\vecf{f}\left(\fnt{u}_{j}\right)}{2}\cdot\frac{\fnt{n}_{ij}}{\left\lVert\fnt{n}_{ij}\right\rVert}\right),\end{align*}
where $\fnt{n}_{ij}\in\mathbb{R}^{d}$ is defined such that $\fnt{n}_{ij,k}=2\fnt{D}_{k,ij}$, and $\vecf{f}:\mathbb{R}^{v}\to\mathbb{R}^{v\times d}$, such that the $k$th column $\left(\vecf{f}\left(\vecf{u}\right)\right)_{k}=\vecf{f}_{k}\left(\vecf{u}\right)$. The relation on \eqref{eq:FluxDiffL1} arises from the property that $\fnt{D}_{k}\fnt{1}=\fnt{0}$. Importantly, 
\begin{align*}\vecf{f}^{\text{central}}\left(\fnt{u}_{i},\fnt{u}_{j},\hat{\fnt{n}}_{ij}\right)=\left(\frac{\vecf{f}\left(\fnt{u}_{i}\right)+\vecf{f}\left(\fnt{u}_{j}\right)}{2}\cdot\hat{\fnt{n}}_{ij}\right),\labell{eq:CentralFlux}\end{align*}
is known as the \textit{central flux}. We can therefore generalize the discretization of the conservation law \eqref{eq:ConservationLaw} as such:
\begin{align*}\fnt{M}\frac{\text{d}\fnt{u}^{H}}{\text{d}t}+\fnt{r}^{H}\aligned{=}\fnt{0},\\
\fnt{r}_{i}^{H}\aligned{=}\sum_{j}^{ }\left\lVert\fnt{n}_{ij}\right\rVert\vecf{f}^{\text{vol}}\left(\fnt{u}_{i},\fnt{u}_{j},\frac{\fnt{n}_{ij}}{\left\lVert\fnt{n}_{ij}\right\rVert}\right),\labell{eq:FluxDiff}\end{align*}
where $\fnt{r}^{H}$ is called the \textit{high order volume term}, and $\vecf{f}^{\text{vol}}$ is a suitable volume flux, which satisfies the two following properties:
\begin{align*}\vecf{f}^{\text{vol}}\left(\fnt{u},\fnt{u},\hat{\fnt{n}}\right)\aligned{=}\vecf{f}\left(\fnt{u}\right)\cdot\hat{\fnt{n}}&\text{Consistency},\\
\vecf{f}^{\text{vol}}\left(\fnt{u}_{i},\fnt{u}_{j},\hat{\fnt{n}}_{ij}\right)\aligned{=}-\vecf{f}^{\text{vol}}\left(\fnt{u}_{j},\fnt{u}_{i},-\hat{\fnt{n}}_{ij}\right)&\text{Skew-Symmetry}\labell{eq:VolumeFlux}.\end{align*}
We call this procedure \textit{flux differencing} \cite{Gassner16, Chen17}. The properties of $\vecf{f}^{\text{vol}}$ determine the properties of the resulting solution. If the volume flux is the central flux, the high order accurate scheme \eqref{eq:HOM} is recovered. If the volume flux is entropy stable or entropy conservative as in \eqref{eq:EntropyStable}, the resulting scheme \eqref{eq:FluxDiff} is entropy stable or entropy conservative, respectively. For instance, the local Lax-Friedrichs flux (LxF)
\begin{align*}\vecf{f}^{\text{LxF}}\left(\fnt{u}_{i},\fnt{u}_{j},\hat{\fnt{n}}_{ij}\right)=\vecf{f}^{\text{central}}\left(\fnt{u}_{i},\fnt{u}_{j},\hat{\fnt{n}}_{ij}\right)+\frac{\lambda_{ij}}{2}\left(\fnt{u}_{i}-\fnt{u}_{j}\right),\labell{eq:LxF}\end{align*}
where $\lambda_{ij}$ is the maximum wavespeed of the 1D Riemann problem with $\fnt{u}_{i}$ and $\fnt{u}_{j}$, is a low order, dissipative, entropy stable flux \cite{Guermond19, Pazner21, Lin23, RuedaRamirez23}. 

\section{Entropy Stable Finite Difference Methods}\label{Entropy Stable Finite Difference Methods}

Entropy stable schemes are numerical methods which satisfy a form of entropy inequality. The inequality is derived so that a discrete form of the 2nd law of thermodynamics is satisfied by the scheme. The entropy inequality is a generalization of an energy inequality for nonlinear conservation laws. Assuming admissible quantities, such as positive pressure and density for the compressible Euler equations, a scheme which satisfies a (semi) discrete entropy inequality typically behaves more robustly. The result is a solution which dissipates mathematical entropy over time, which in some cases is even sufficient to stabilize the solution without relying on heuristic stabilization techniques. 

We are particularly interested in equations which satisfy an entropy inequality. The continuous entropy inequality we wish to mimic at the discrete level is:
\begin{align*}\frac{\partial\eta\left(\vecf{u}\right)}{\partial t}+\sum_{k=1}^{d}\frac{\partial F_{k}\left(\vecf{u}\right)}{\partial x_{k}}\le0,\labell{eq:EntropyInequality}\end{align*}
where $\eta:\mathbb{R}^{v}\to\mathbb{R}$ is called the \textit{entropy function}, and $F_{1},...,F_{d}:\mathbb{R}^{v}\to\mathbb{R}$ are called the \textit{entropy fluxes}. In particular, we choose $\eta$ and $F_{k}$ such that $\left(\eta,F_{k}\right)$ are a \textit{convex entropy, entropy flux pair}. Here, $\eta$ is convex, and
\begin{align*}\frac{\partial F_{k}}{\partial\vecf{u}}=\vecf{v}^{T}\frac{\partial\vecf{f}_{k}\left(\vecf{u}\right)}{\partial\vecf{u}},\end{align*}
where $\vecf{v}=\nabla_{\vecf{u}}\eta$ are referred to as the \textit{entropy variables}. The \textit{entropy fluxes} $F_{k}$ are given by
\begin{align*}F_{k}\left(\vecf{u}\right)=\vecf{v}^{T}\vecf{f}_{k}\left(\vecf{u}\right)-\psi_{k}\left(\vecf{u}\right),\labell{eq:EntropyPotential}\end{align*}
where $\psi_{1},...,\psi_{d}:\mathbb{R}^{v}\to\mathbb{R}$ are referred to as the \textit{entropy potentials} \cite{Godlewski13}. We define an entropy \textit{stable} or \textit{conservative} flux as one which satisfies one of the following properties:
\begin{align*}\left(\fnt{v}_{j}-\fnt{v}_{i}\right)^{T}\vecf{f}\left(\fnt{u}_{i},\fnt{u}_{j},\hat{\fnt{n}}_{ij}\right)\aligned{=}\left(\psi\left(\fnt{u}_{j}\right)-\psi\left(\fnt{u}_{i}\right)\right)\cdot\hat{\fnt{n}}_{ij}&\text{Entropy Conservative},\\
\left(\fnt{v}_{j}-\fnt{v}_{i}\right)^{T}\vecf{f}\left(\fnt{u}_{i},\fnt{u}_{j},\hat{\fnt{n}}_{ij}\right)\aligned{\le}\left(\psi\left(\fnt{u}_{j}\right)-\psi\left(\fnt{u}_{i}\right)\right)\cdot\hat{\fnt{n}}_{ij}&\text{Entropy Stable}.\labell{eq:EntropyStable}\end{align*}
Recall, if a volume flux $\vecf{f}^{\text{vol}}$ is entropy stable, the flux differencing scheme \eqref{eq:FluxDiff} satisfies an entropy inequality \eqref{eq:EntropyInequality}. For instance, the low order scheme \eqref{eq:LOM} with entropy stable flux $\vecf{f}^{\text{LxF}}$ satisfies the entropy inequality \eqref{eq:EntropyInequality}. 

\subsection{Discrete Nodal Entropy Stability}\label{Discrete Nodal Entropy Stability}

In this section, we will discuss the discretization of a high order accurate (in sufficiently regular regions), conservative, entropy stable scheme of the form:
\begin{align*}\fnt{M}\frac{\text{d}\fnt{u}}{\text{d}t}+\fnt{r}\aligned{=}\fnt{0},\\
\fnt{r}_{i}\aligned{=}\sum_{j}^{ }\left\lVert\fnt{n}_{ij}\right\rVert\vecf{f}_{ij}\left(\fnt{\theta}_{ij}\right)\labell{eq:BlendedScheme}.\end{align*}
Here, $\fnt{\theta}\ge0$ are called the \textit{diffusion coefficients}, which parameterize $\vecf{f}_{ij}\left(\fnt{\theta}_{ij}\right)$ via
\begin{align*}\vecf{f}_{ij}\left(\fnt{\theta}\right)=\vecf{f}_{ij}^{H}+\fnt{\theta}\left(\fnt{u}_{i}-\fnt{u}_{j}\right),\labell{eq:AVFlux}\end{align*}
where $\vecf{f}_{ij}^{H}=\vecf{f}^{\text{central}}\left(\fnt{u}_{i},\fnt{u}_{j},\hat{\fnt{n}}_{ij}\right)$ is the high order accurate central flux. Assuming $\fnt{\theta}_{ij}$ is symmetric, this is equivalent to adding a weighted graph Laplacian term to the formulation through the high order flux. That is, $\fnt{\theta}_{ij}$ is proportional to dissipation introduced into the flux $\vecf{f}_{ij}^{H}$.

We will now discuss a high order accurate (in sufficiently smooth regions) discretization of the entropy inequality \eqref{eq:EntropyInequality} on node $x_{i}$. To derive an entropy balance, we first apply the entropy variables $\fnt{v}_{i}$ to the semi discrete scheme \eqref{eq:BlendedScheme}:
\begin{align*}\fnt{M}_{ii}\fnt{v}_{i}^{T}\frac{\text{d}\fnt{u}_{i}}{\text{d}t}+\fnt{v}_{i}^{T}\sum_{j}^{ }\left\lVert\fnt{n}_{ij}\right\rVert\vecf{f}_{ij}\left(\fnt{\theta}_{ij}\right)=0.\labell{eq:viDot}\end{align*}
This entropy balance allows for the derivation of a discrete entropy inequality parameterized by the diffusion coefficients. These diffusion coefficients $\fnt{\theta}_{ij}$ can be chosen to enforce a discrete entropy inequality. To do this, the diffusion coefficients will be restricted to satisfy the following inequality:
\begin{align*}\sum_{j}^{ }\left\lVert\fnt{n}_{ij}\right\rVert\left[\left(\fnt{v}_{j}-\fnt{v}_{i}\right)^{T}\vecf{f}_{ij}\left(\fnt{\theta}_{ij}\right)-\left(\psi\left(\fnt{u}_{j}\right)-\psi\left(\fnt{u}_{i}\right)\right)\cdot\hat{\fnt{n}}_{ij}\right]\aligned{\le}0.\labell{eq:LooseTheta}\end{align*}
By \eqref{eq:viDot}, the entropy variables $\fnt{v}_i$ can be moved out of the expression, leading to
\begin{align*}\fnt{M}_{ii}\fnt{v}_{i}^{T}\frac{\text{d}\fnt{u}_{i}}{\text{d}t}+\sum_{j}^{ }\left\lVert\fnt{n}_{ij}\right\rVert\left(\fnt{v}_{j}^{T}\vecf{f}_{ij}\left(\fnt{\theta}_{ij}\right)-\left(\psi\left(\fnt{u}_{j}\right)-\psi\left(\fnt{u}_{i}\right)\right)\cdot\hat{\fnt{n}}_{ij}\right)\aligned{\le}0.\end{align*}
Note, this inequality also arises from a discretization of the entropy inequality \eqref{eq:EntropyInequality} on node $x_{i}$, where flux differencing is used to discretize the spatial derivative of the entropy flux. Therefore, by restricting the diffusion coefficients $\fnt{\theta}_{ij}$ to the condition \eqref{eq:LooseTheta}, we are satisfying a semi-discrete entropy inequality. To derive a more explicit condition on the diffusion coefficients, we revisit \eqref{eq:LooseTheta}. Directly plugging in the definition of $\vecf{f}_{ij}\left(\fnt{\theta}_{ij}\right)$ from \eqref{eq:AVFlux}, we obtain:
\begin{align*}\sum_{j}^{ }\left\lVert\fnt{n}_{ij}\right\rVert\left[\left(\fnt{v}_{j}-\fnt{v}_{i}\right)^{T}\left(\vecf{f}_{ij}^{H}+\fnt{\theta}_{ij}\left(\fnt{u}_{i}-\fnt{u}_{j}\right)\right)-\left(\psi\left(\fnt{u}_{j}\right)-\psi\left(\fnt{u}_{i}\right)\right)\cdot\hat{\fnt{n}}_{ij}\right]\aligned{\le}0,\\
\sum_{j}^{ }\left\lVert\fnt{n}_{ij}\right\rVert\left[\left(\fnt{v}_{j}-\fnt{v}_{i}\right)^{T}\vecf{f}_{ij}^{H}-\left(\psi\left(\fnt{u}_{j}\right)-\psi\left(\fnt{u}_{i}\right)\right)\cdot\hat{\fnt{n}}_{ij}\right]-\sum_{j}^{ }\left\lVert\fnt{n}_{ij}\right\rVert\left[\left(\fnt{v}_{j}-\fnt{v}_{i}\right)^{T}\left(\fnt{u}_{j}-\fnt{u}_{i}\right)\right]\fnt{\theta}_{ij}\aligned{\le}0,\\
b_{i}-\fnt{a}_{i}^{T}\fnt{\theta}_{i}\le0,\labell{eq:DiscreteEI}\end{align*}
where
\begin{align*}\left(\fnt{a}_{i}\right)_{j}=\fnt{a}_{ij}\aligned{=}\left\lVert\fnt{n}_{ij}\right\rVert\left[\left(\fnt{v}_{j}-\fnt{v}_{i}\right)^{T}\left(\fnt{u}_{j}-\fnt{u}_{i}\right)\right],\\
b_{i}\aligned{=}\sum_{j}^{ }\left\lVert\fnt{n}_{ij}\right\rVert\left[\left(\fnt{v}_{j}-\fnt{v}_{i}\right)^{T}\vecf{f}_{ij}^{H}-\left(\psi\left(\fnt{u}_{j}\right)-\psi\left(\fnt{u}_{i}\right)\right)\cdot\hat{\fnt{n}}_{ij}\right],\end{align*}
and $\left(\fnt{\theta}_{i}\right)_{j}=\fnt{\theta}_{ij}$. 

By \eqref{eq:EntropyStable}, one can observe that if the high order flux $\vecf{f}_{ij}^H$ is entropy stable at $\fnt{u}_i$, then $b_i \leq 0$ and hence $\fnt{\theta}_i$ satisfies the semi-discrete entropy inequality \eqref{eq:DiscreteEI}. If $b_i>0$, then $\fnt{a}_i^T \fnt{\theta}_i$ can be used to apply local dissipation according to \eqref{eq:DiscreteEI}.

The discrete entropy inequality \eqref{eq:DiscreteEI} allows one to determine diffusion coefficients $\fnt{\theta}_{ij}$ such that the resulting scheme \eqref{eq:BlendedScheme} satisfies the semi-discrete entropy inequality. In the next lemma, we argue that such a discrete entropy inequality is feasible.

\begin{lemma}
\label{lemma:Feasibility}
Assume the finite difference solution $\fnt{u}$ is admissible, such that $\eta$ is strictly convex at each $\fnt{u}_i$. Then, there exists $\fnt{\theta}_{i}\ge0$ such that $\fnt{a}_{i}^{T}\fnt{\theta}_{i}\ge b_{i}$.
\end{lemma}
\begin{proof}
The entropy projection matrix $\left(\frac{\text{d}\fnt{u}}{\text{d}\fnt{v}}\right)_{\fnt{u}}\in\mathbb{R}^{v\times v}$, evaluated at any admissible intermediate state between $\fnt{u}_{i}$ and $\fnt{u}_{j}$, is symmetric positive definite. Therefore, if we consider $\fnt{a}_{ij}$:
\begin{align*}\fnt{a}_{ij}\aligned{=}\left\lVert\fnt{n}_{ij}\right\rVert\left[\left(\fnt{v}_{j}-\fnt{v}_{i}\right)^{T}\left(\fnt{u}_{j}-\fnt{u}_{i}\right)\right],\\
\aligned{=}\left\lVert\fnt{n}_{ij}\right\rVert\left[\left(\fnt{v}_{j}-\fnt{v}_{i}\right)^{T}\left(\frac{\text{d}\fnt{u}}{\text{d}\fnt{v}}\right)_{\fnt{u}}\left(\fnt{v}_{j}-\fnt{v}_{i}\right)\right],\tag{$\text{chain rule}$}\\
\aligned{\ge}0.\end{align*}
In the case where $b_{i}\le0$, $\fnt{\theta}_{i}=\fnt{0}$ is feasible for \eqref{eq:DiscreteEI}. Suppose $b_{i}>0$. In this case, since $b_{i}$ is nonzero, there must be some $j$ for which $\left\lVert\fnt{n}_{ij}\right\rVert\left[\left(\fnt{v}_{j}-\fnt{v}_{i}\right)^{T}\vecf{f}_{ij}^{H}-\left(\psi\left(\fnt{u}_{j}\right)-\psi\left(\fnt{u}_{i}\right)\right)\cdot\hat{\fnt{n}}_{ij}\right]>0$. Thus, we must also have that $\fnt{u}_{i}\ne\fnt{u}_{j}$. Therefore, $\fnt{a}_{ij}>0$, and so $\fnt{a}_{i}\ne\fnt{0}$. Thus, $\fnt{a}_{i}^{T}\fnt{a}_{i}>0$. Fixing $\lambda=\frac{b_{i}}{\fnt{a}_{i}^{T}\fnt{a}_{i}}>0$, we see that $\fnt{\theta}_{i}=\lambda\fnt{a}_{i}\ge0$ satisfies $\fnt{a}_{i}^{T}\fnt{\theta}_{i}=b_{i}$. 
\end{proof}
Consider the form of $\vecf{f}_{ij}\left(\fnt{\theta}_{ij}\right)$ from \eqref{eq:AVFlux}. Recall, $\fnt{\theta}_{ij}$ is proportional to dissipation added to $\vecf{f}_{ij}^{H}$. So, we seek to minimize $\fnt{\theta}_{i}$ in some sense while ensuring that it satisfies $\fnt{a}_{i}^{T}\fnt{\theta}_{i}\ge b_{i}$ and $\fnt{\theta}_{i}\ge0$. Note that the expression $\fnt{\theta}_{i}\ge0$ is an abuse of notation to denote $\forall j, \fnt{\theta}_{ij} \geq 0$. To do this, we will set the \textit{pre-symmetrized} diffusion coefficients $\hat{\fnt{\theta}}_{i}$ as the solution to the optimization problem:
\begin{align*}\min_{\begin{matrix}\fnt{a}_{i}^{T}\hat{\fnt{\theta}}\ge b_{i}\\
\hat{\fnt{\theta}}\ge0\end{matrix}}\hat{\fnt{\theta}}^{T}\hat{\fnt{\theta}}.\labell{eq:Knapsack}\end{align*}
Note that, in order to preserve conservation (see Section \eqref{Ensuring Conservation} for more details), the diffusion coefficients must undergo an additional symmetrization procedure before being applied to the blended scheme \eqref{eq:BlendedScheme}.

The following lemma will show that such a problem has an explicit form for its optimal solution.

\begin{lemma}
\label{lemma:Knapsack Solution}
The problem \eqref{eq:Knapsack} is feasible and has optimal solution $\hat{\fnt{\theta}}=\frac{b_{i}\fnt{a}_{i}}{\fnt{a}_{i}^{T}\fnt{a}_{i}}$ if $b_{i}>0$, and $\hat{\fnt{\theta}}=\fnt{0}$ if $b_{i}\le0$. 
\end{lemma}
\begin{proof}
By Lemma \eqref{lemma:Feasibility}, the optimization problem \eqref{eq:Knapsack} is feasible. If $b_{i}\le0$, then, $\hat{\fnt{\theta}}=\fnt{0}$ is feasible. Moreover, $\fnt{0}$ optimizes the objective $\hat{\fnt{\theta}}^{T}\hat{\fnt{\theta}}=\left\lVert\hat{\fnt{\theta}}\right\rVert^{2}$ over its whole domain. Therefore, $\hat{\fnt{\theta}}=\fnt{0}$ is optimal when $b_{i}\le0$. 

Suppose $b_{i}>0$. Consider the relaxed form of the optimization problem \eqref{eq:Knapsack}:
\begin{align*}\min_{\fnt{a}_{i}^{T}\hat{\fnt{\theta}}\ge b_{i}}\hat{\fnt{\theta}}^{T}\hat{\fnt{\theta}}.\labell{eq:Relaxed}\end{align*}
Such a problem has a well-known explicit solution: $\hat{\fnt{\theta}}_{\ast}=\fnt{a}_{i}^{\dagger}b_{i}=\frac{b_{i}\fnt{a}_{i}}{\fnt{a}_{i}^{T}\fnt{a}_{i}}$. Recall from Lemma \eqref{lemma:Feasibility} that $\fnt{a}_{i}\ge0$. Therefore, we have that $\hat{\fnt{\theta}}_{\ast}\ge0$. Hence, $\hat{\fnt{\theta}}_{\ast}$ is feasible for \eqref{eq:Knapsack}. Since it is optimal for the relaxed form \eqref{eq:Relaxed}, it also optimizes \eqref{eq:Knapsack}. 
\end{proof}

Note that the solution to the optimization problem is cheap to compute. The resulting scheme does not dramatically increase the runtime of standard finite difference methods. 

\subsection{Ensuring Conservation}\label{Ensuring Conservation}

In this section, we discuss a slight restriction on the resulting scheme in order to enforce \textit{conservation}. Conservation is an important property for nonlinear conservation laws which ensures that the variables $\fnt{u}$ remain discretely conserved by our scheme. For periodic boundary conditions, a scheme is called conservative if the discrete volume integral is equal to $\fnt{0}$:
\begin{align*}\sum_{i}^{ }\fnt{M}_{ii}\frac{\text{d}\fnt{u}_{i}}{\text{d}t}=\fnt{0}.\labell{eq:Conservative}\end{align*}
We will also show that the resulting scheme with symmetrized diffusion coefficients is provably entropy stable. For the purposes of conservation, we will require that $\fnt{\theta}_{ij}=\fnt{\theta}_{ji}$. This can be obtained by setting the diffusion coefficients $\fnt{\theta}_{ij}=\max\left\{\hat{\fnt{\theta}}_{ij},\hat{\fnt{\theta}}_{ji}\right\}$ after determining $\hat{\fnt{\theta}}_{ij}$ from \eqref{eq:Knapsack}. \footnote{Importantly, one may instead optimize the diffusion coefficients subject to a symmetry condition. However, doing so necessarily requires that all nodal entropy inequalities across the domain be conditions to the same optimization problem, sacrificing parallelism. Furthermore, such a problem does not admit a closed form, and existing algorithms for solving for an optimal solution admit a runtime which exceeds that of a step of finite difference.} The next lemma argues that this maximization procedure does not affect the entropy stability of the resulting scheme. 

\begin{lemma}
\label{lemma:Symmetrization}
The blended scheme \eqref{eq:BlendedScheme} is entropy stable when applied with the diffusion coefficients $\fnt{\theta}_{ij}=\max\left\{\hat{\fnt{\theta}}_{ij},\hat{\fnt{\theta}}_{ji}\right\}$.
\end{lemma}
\begin{proof}
Consider node $x_{i}$. Notice that $\fnt{\theta}_{ij}\ge\hat{\fnt{\theta}}_{ij}$ by definition. Therefore, since $\fnt{a}_{i}\ge0$, it must be the case that $\fnt{a}_{i}^{T}\fnt{\theta}_{i}\ge\fnt{a}_{i}^{T}\hat{\fnt{\theta}}_{i}\ge b_{i}$. Therefore, it satisfies \eqref{eq:DiscreteEI}. Hence, the blended scheme is provably entropy stable.

\end{proof}

In order to satisfy conservation for periodic boundary conditions, it suffices to show that the discrete volume integral is equal to $\fnt{0}$. This next lemma will argue that the imposition of symmetry $\fnt{\theta}_{ij}=\fnt{\theta}_{ji}$ is sufficient for conservation in periodic boundary conditions.

\begin{lemma}
\label{lemma:Periodic Conservation}
The blended scheme \eqref{eq:BlendedScheme} is conservative when applied with the diffusion coefficients $\fnt{\theta}_{ij}$.
\end{lemma}
\begin{proof}

Note, in the context of periodic boundary conditions, $\fnt{D}_{k}$ is skew-symmetric. Thus, $\fnt{n}_{ij}=-\fnt{n}_{ji}$. Furthermore, since $\fnt{\theta}_{ij}=\fnt{\theta}_{ji}$ we have from \eqref{eq:AVFlux} that $\vecf{f}_{ij}\left(\fnt{\theta}_{ij}\right)=-\vecf{f}_{ji}\left(\fnt{\theta}_{ji}\right)$. Therefore,
\begin{align*}\sum_{i}^{ }\fnt{M}_{ii}\frac{\text{d}\fnt{u}_{i}}{\text{d}t}\aligned{=}-\sum_{i}^{ }\sum_{j}^{ }\left\lVert\fnt{n}_{ij}\right\rVert\vecf{f}_{ij}\left(\fnt{\theta}_{ij}\right),\tag{$\text{by}\ \eqref{eq:BlendedScheme}$}\\
\aligned{=}\sum_{i}^{ }\sum_{j}^{ }\left\lVert-\fnt{n}_{ji}\right\rVert\vecf{f}_{ji}\left(\fnt{\theta}_{ji}\right),\\
\aligned{=}-\sum_{i}^{ }\fnt{M}_{ii}\frac{\text{d}\fnt{u}_{i}}{\text{d}t}.\end{align*}
Therefore, $\sum_{i}^{ }\fnt{M}_{ii}\frac{\text{d}\fnt{u}_{i}}{\text{d}t}=\fnt{0}$. Therefore, by \eqref{eq:Conservative}, the scheme \eqref{eq:BlendedScheme} is conservative.
\end{proof}

\subsection{Weakly-Enforced Boundary Conditions}\label{Weakly-Enforced Boundary Conditions}

In this section, we discuss the arising discretization when enforcing constant-in-time Dirichlet boundary conditions weakly. Note that strong boundary conditions may be enforced as well \cite{Svard25}. In order to do this, we utilize the SBP property that $\fnt{D}_{k}+\fnt{D}_{k}^{T}=\fnt{B}_{k}$, where $\fnt{B}_{k}$ is a diagonal boundary matrix. Recall the flux differencing technique taken in \eqref{High Order Finite Difference Methods}. Applying a similar flux differencing technique, we obtain:

\begin{align*}\sum_{k=1}^{d}\left(\fnt{D}_{k}\vecf{f}_{k}\left(\fnt{u}\right)\right)_{i}\aligned{=}\sum_{k=1}^{d}\sum_{j}^{ }2\fnt{D}_{k,ij}\left(\frac{\vecf{f}_{k}\left(\fnt{u}_{i}\right)+\vecf{f}_{k}\left(\fnt{u}_{j}\right)}{2}\right),\\
\aligned{=}\sum_{k=1}^{d}\sum_{j}^{ }\left(\fnt{B}_{k,ij}+\fnt{D}_{k,ij}-\fnt{D}_{k,ij}^{T}\right)\cdot\left(\frac{\vecf{f}\left(\fnt{u}_{i}\right)+\vecf{f}\left(\fnt{u}_{j}\right)}{2}\right),\labell{eq:FluxDiffL2}\\
\aligned{=}\sum_{j}^{ }\left\lVert\fnt{n}_{ij}\right\rVert\cdot\left(\frac{\vecf{f}\left(\fnt{u}_{i}\right)+\vecf{f}\left(\fnt{u}_{j}\right)}{2}\cdot\frac{\fnt{n}_{ij}}{\left\lVert\fnt{n}_{ij}\right\rVert}\right)+\left\lVert\fnt{b}_{i}\right\rVert\left(\vecf{f}\left(\fnt{u}_{i}\right)\cdot\frac{\fnt{b}_{i}}{\left\lVert\fnt{b}_{i}\right\rVert}\right),\end{align*}
where $\fnt{n}_{ij},\fnt{b}_{i}\in\mathbb{R}^{d}$ such that $\fnt{n}_{ij,k}=\left(\fnt{D}_{k}-\fnt{D}_{k}^{T}\right)_{ij}$ and $\fnt{b}_{i,k}=\fnt{B}_{k,ii}$ are outward unit normals scaled by the surface quadrature weights and surface Jacobian factor. The relation on \eqref{eq:FluxDiffL2} arises from the property that $2\fnt{D}_{k}=\fnt{B}_{k}+\fnt{D}_{k}-\fnt{D}_{k}^{T}$. Therefore, the high order method \eqref{eq:FluxDiff} and the blended scheme \eqref{eq:BlendedScheme} can alternatively be discretized by
\begin{align*}\fnt{M}\frac{\text{d}\fnt{u}_{i}^{H}}{\text{d}t}+\sum_{j}^{ }\left\lVert\fnt{n}_{ij}\right\rVert\vecf{f}_{ij}^{H}=-\left\lVert\fnt{b}_{i}\right\rVert\left(\vecf{f}^{\text{surf}}\left(\fnt{u}_{i},\fnt{u}_{i}^{\text{BC}},\frac{\fnt{b}_{i}}{\left\lVert\fnt{b}_{i}\right\rVert}\right)\right),\labell{eq:wbcFluxDiff}\\
\fnt{M}\frac{\text{d}\fnt{u}_{i}}{\text{d}t}+\sum_{j}^{ }\left\lVert\fnt{n}_{ij}\right\rVert\vecf{f}_{ij}\left(\fnt{\theta}_{ij}\right)=-\left\lVert\fnt{b}_{i}\right\rVert\left(\vecf{f}^{\text{surf}}\left(\fnt{u}_{i},\fnt{u}_{i}^{\text{BC}},\frac{\fnt{b}_{i}}{\left\lVert\fnt{b}_{i}\right\rVert}\right)\right),\labell{eq:wbcBlendedScheme}\end{align*}
where $\vecf{f}^{\text{surf}}$ is a chosen surface flux. The choice of the surface flux is arbitrary, but it must be entropy stable, as in \eqref{eq:EntropyStable}. By applying the same procedure in Section \eqref{Discrete Nodal Entropy Stability}, one arrives at the alternative formulation of a discrete entropy inequality:
\begin{align*}\aligned{ }\left\lVert\fnt{b}_{i}\right\rVert\left(-\fnt{v}_{i}^{T}\vecf{f}^{\text{surf}}\left(\fnt{u}_{i},\fnt{u}_{i}^{\text{BC}},\frac{\fnt{b}_{i}}{\left\lVert\fnt{b}_{i}\right\rVert}\right)-\psi\left(\fnt{u}_{i}\right)\cdot\frac{\fnt{b}_{i}}{\left\lVert\fnt{b}_{i}\right\rVert}\right)\labell{eq:GrossBoundaryES}\\
\aligned{+}\sum_{j}^{ }\left\lVert\fnt{n}_{ij}\right\rVert\left[\left(\fnt{v}_{j}-\fnt{v}_{i}\right)^{T}\vecf{f}_{ij}\left(\fnt{\theta}_{ij}\right)-\left(\psi\left(\fnt{u}_{j}\right)-\psi\left(\fnt{u}_{i}\right)\right)\cdot\hat{\fnt{n}}_{ij}\right]\le0.\labell{eq:wbcDiscreteEI}\end{align*}
Under appropriate boundary conditions, the first term \eqref{eq:GrossBoundaryES} can be made entropy stable \cite{Chen17, Svard21}. Therefore, a sufficient condition for semi-discrete entropy stability is the entropy inequality on \eqref{eq:wbcDiscreteEI}. Observe that this inequality is equivalent in form to that in \eqref{eq:DiscreteEI}. That is, satisfying the discrete entropy inequality \eqref{eq:DiscreteEI} results in an entropy stable scheme, even in the context of weakly enforced boundary conditions. 

In the case of nonperiodic boundary conditions, conservation is implied by the discrete volume integral being equal to the discrete surface integral of the flux function:
\begin{align*}\sum_{i}^{ }\fnt{M}_{ii}\frac{\text{d}\fnt{u}_{i}}{\text{d}t}=-\sum_{i}^{ }\left\lVert\fnt{b}_{i}\right\rVert\left(\vecf{f}^{\text{surf}}\left(\fnt{u}_{i},\fnt{u}_{i}^{\text{BC}},\frac{\fnt{b}_{i}}{\left\lVert\fnt{b}_{i}\right\rVert}\right)\right).\end{align*}
In order to prove the resulting scheme is conservative, it suffices to show that the discrete integral is equal to the boundary integral of the fluxes:
\begin{align*}\aligned{ }\sum_{i}^{ }\fnt{M}_{ii}\frac{\text{d}\fnt{u}_{i}}{\text{d}t},\\
\aligned{=}-\sum_{i}^{ }\sum_{j}^{ }\left\lVert\fnt{n}_{ij}\right\rVert\vecf{f}_{ij}\left(\fnt{\theta}_{ij}\right)-\sum_{i}^{ }\left\lVert\fnt{b}_{i}\right\rVert\left(\vecf{f}^{\text{surf}}\left(\fnt{u}_{i},\fnt{u}_{i}^{\text{BC}},\frac{\fnt{b}_{i}}{\left\lVert\fnt{b}_{i}\right\rVert}\right)\right),\tag{$\text{by}\ \eqref{eq:BlendedScheme}$}\\
\aligned{=}-\sum_{i}^{ }\left\lVert\fnt{b}_{i}\right\rVert\left(\vecf{f}^{\text{surf}}\left(\fnt{u}_{i},\fnt{u}_{i}^{\text{BC}},\frac{\fnt{b}_{i}}{\left\lVert\fnt{b}_{i}\right\rVert}\right)\right).\tag{$\text{see}\ \eqref{lemma:Periodic Conservation}$}\end{align*}
\section{Equivalence with Knapsack Limiting and Positivity Preservation}\label{Equivalence with Knapsack Limiting and Positivity Preservation}

In this section, we will tie a connection between knapsack limiting and the ECAV-FD scheme, and we will discuss a strategy for enforcing positivity preservation into the proposed scheme. To do this, we will assume that the chosen timestepper is strong stability preserving (SSP). The update of an SSP timestepper can be expressed as a convex combination of forward Euler updates, which combines the benefits of positivity preservation from forward Euler timestepping with low order methods, with the high order time converges of Runge-Kutta methods. Consider the low order method defined as follows:
\begin{align*}\fnt{M}\frac{\text{d}\fnt{u}^{L}}{\text{d}t}+\fnt{r}^{L}\aligned{=}\fnt{0},\\
\fnt{r}_{i}^{L}\aligned{=}\sum_{j}^{ }\left\lVert\fnt{n}_{ij}\right\rVert\vecf{f}_{ij}^{L},\labell{eq:LOM}\\
\vecf{f}_{ij}^{L}\aligned{=}\vecf{f}^{\text{LxF}}\left(\fnt{u}_{i},\fnt{u}_{j},\frac{\fnt{n}_{ij}}{\left\lVert\fnt{n}_{ij}\right\rVert}\right).\labell{eq:LowOrderFlux}\end{align*}
In order to enforce constant-in-time Dirichlet boundary conditions, one may replace the right-hand-side with the appropriate boundary term. The resulting low order method is provably positivity preserving for the compressible Euler equations under appropriate choices of the timestep \cite{Lin23}. Notice:
\begin{align*}\vecf{f}_{ij}^{L}\aligned{=}\vecf{f}_{ij}^{H}+\left(\vecf{f}_{ij}^{L}-\vecf{f}_{ij}^{H}\right),\tag{$\text{see}\ \eqref{eq:LowOrderFlux}\ \text{and}\ \eqref{eq:CentralFlux}$}\\
\aligned{=}\vecf{f}_{ij}^{H}+\frac{\lambda_{ij}}{2}\left(\fnt{u}_{i}-\fnt{u}_{j}\right),\tag{$\text{by}\ \eqref{eq:LxF}$}\\
\aligned{=}\vecf{f}_{ij}\left(\frac{\lambda_{ij}}{2}\right),\end{align*}
where $\lambda_{ij}$ is defined as in \eqref{eq:LxF}. We can see from the definition above that there exists a set of diffusion coefficients such that the scheme \eqref{eq:wbcBlendedScheme} is equivalent to the low order scheme \eqref{eq:LOM}. Furthermore, more generally, we can observe that for $\fnt{\theta}_{ij}\in\left[0,1\right]$:
\begin{align*}\vecf{f}_{ij}\left(\frac{\lambda_{ij}}{2}\fnt{\theta}_{ij}\right)\aligned{=}\vecf{f}_{ij}^{H}+\frac{\lambda_{ij}}{2}\fnt{\theta}_{ij}\left(\fnt{u}_{i}-\fnt{u}_{j}\right),\\
\aligned{=}\vecf{f}_{ij}^{H}+\fnt{\theta}_{ij}\left(\vecf{f}_{ij}^{L}-\vecf{f}_{ij}^{H}\right),\\
\aligned{=}\left(1-\fnt{\theta}_{ij}\right)\vecf{f}_{ij}^{H}+\fnt{\theta}_{ij}\vecf{f}_{ij}^{L}.\end{align*}
That is, $\vecf{f}_{ij}\left(\frac{\lambda_{ij}}{2}\fnt{\theta}_{ij}\right)$ can be expressed as a convex combination between the low and high order fluxes $\vecf{f}_{ij}^{L}$ and $\vecf{f}_{ij}^{H}$. We then define the weighted evaluation of the flux \eqref{eq:AVFlux}, $\vecf{f}_{ij}^{\text{KL}}\left(\fnt{\theta}_{ij}\right)=\vecf{f}_{ij}\left(\frac{\lambda_{ij}}{2}\fnt{\theta}_{ij}\right)$ the \textit{knapsack} blended flux. The use of this knapsack flux in the scheme \eqref{eq:wbcBlendedScheme} amounts to \textit{knapsack limiting}, which can be used to preserve positivity. \textit{Limiting coefficients} $\fnt{\ell^{c}}$ can be derived such that if $\fnt{\theta}_{ij}\in\left[\fnt{\ell^{c}}_{ij},1\right]$, the resulting scheme satisfies a \textit{relative positivity condition}:
\begin{align*}\fnt{u}_{i}+\Delta t\cdot\frac{\text{d}\fnt{u}_{i}}{\text{d}t}\ge\alpha\left(\fnt{u}_{i}+\Delta t\cdot\frac{\text{d}\fnt{u}_{i}^{L}}{\text{d}t}\right),\labell{eq:RelativePositivity}\end{align*}
for some $\alpha\in\left[0,1\right]$. In the event that the solution is close to violating the relative positivity condition \eqref{eq:RelativePositivity} at node $x_i$, the corresponding limiting coefficient $\fnt{\ell^c}$ specifies a lower bound on the amount of extra dissipation added to the solution to ensure positivity at $\fnt{u}_i$. Importantly, it must be the case that $\fnt{\ell^{c}}_{ij}=\fnt{\ell^{c}}_{ji}$ to ensure conservation is satisfied. Thus, a symmetrization process must take place after determining the limiting coefficients from \eqref{eq:RelativePositivity}. The pre-symmetrization diffusion coefficients $\hat{\fnt{\theta}}_{ij}$ may then be determined from the \textit{quadratic knapsack problem}:
\begin{align*}\min_{\begin{matrix}\fnt{a}_{i}^{T}\hat{\fnt{\theta}}\ge b_{i}\\
0\le\hat{\fnt{\theta}}\le1-\fnt{\ell^{c}}_{i}\end{matrix}}\hat{\fnt{\theta}}^{T}\hat{\fnt{\theta}},\labell{eq:PPKnapsack}\end{align*}
where 
\begin{align*}\left(\fnt{a}_{i}\right)_{j}=\fnt{a}_{ij}\aligned{=}\left\lVert\fnt{n}_{ij}\right\rVert\left[\left(\fnt{v}_{j}-\fnt{v}_{i}\right)^{T}\left(\vecf{f}_{ij}^{H}-\vecf{f}_{ij}^{L}\right)\right],\\
b_{i}\aligned{=}\sum_{j}^{ }\left\lVert\fnt{n}_{ij}\right\rVert\left[\left(\fnt{v}_{j}-\fnt{v}_{i}\right)^{T}\vecf{f}_{ij}^{H}-\left(\psi\left(\fnt{u}_{j}\right)-\psi\left(\fnt{u}_{i}\right)\right)\cdot\hat{\fnt{n}}_{ij}\right]-\fnt{a}_{i}^{T}\fnt{\ell^{c}}_{i},\\
\left(\fnt{\ell^{c}}_{i}\right)_{j}\aligned{=}\fnt{\ell^{c}}_{ij}.\end{align*}
An efficient, linear-time algorithm for this quadratic knapsack problem was proved in \cite{Christner25}. At this point, symmetrizing the diffusion coefficients to $\fnt{\theta}_{ij}=\max\left\{\hat{\fnt{\theta}}_{ij},\hat{\fnt{\theta}}_{ji}\right\}$, the scheme \eqref{eq:wbcBlendedScheme} is provably positivity preserving when applied with the diffusion coefficients $\fnt{\theta}_{ij}+\fnt{\ell^{c}}_{ij}$. 

Note that the use of the Lax-Friedrichs flux for $\vecf{f}^{L}$ is not necessary. The low order flux must be entropy stable, in such a way that the scheme \eqref{eq:LOM} preserves positivity. For instance, the HLLC flux can be chosen instead. However, in this case, the resulting knapsack limiting flux $\vecf{f}_{ij}^{\text{KL}}\left(\fnt{\theta}_{ij}\right)$ is \textit{not} equivalent to a weighted evaluation of the flux \eqref{eq:AVFlux}. That being said, the same strategies apply. 

\section{Numerical Results}\label{Numerical Results}

In this section, we first analyze the spatial convergence of the proposed artificial viscosity (ECAV-FD) and knapsack limiting (KL-FD) schemes, and examine their solutions on an entropy stability benchmark. Then, we determine the performance of each scheme by timing against a standard finite difference method for a smooth initial condition. We then apply the proposed schemes to various positivity preservation benchmarks. Finally, we examine the solution to various two-dimensional problems.

All simulations utilize the compressible Euler equations. Note that the methods in this paper can be easily extended to the Navier-Stokes equations, though it is left here for future work. We use various timesteppers throughout this section, including \texttt{RK4}, a 4-stage 4th order method, and \texttt{SSPRK43}, a strong stability preserving (SSP) 4-stage, 3rd order method. We specify for each experiment whether the timesteppers used a fixed or adaptive timestep, which are both implemented in \texttt{OrdinaryDiffEq.jl} \cite{OrdinaryDiffEq}. Similarly, we use implementations from \texttt{Trixi.jl} \cite{Trixi1, Trixi3} for all quantities related to the compressible Euler equations. The KL-FD scheme utilizes the HLLC flux \cite{Harten83} for its entropy stable flux $\vecf{f}^{L}$ in all simulations. For this reason, we will refer to it as KL-FD-HLLC for clarity.

\subsection{Spatial Convergence}\label{Spatial Convergence}

In this section, we examine the spatial convergence of the ECAV-FD and KL-FD-HLLC schemes in order to demonstrate that they obtain arbitrarily high spatial orders of accuracy. We also visually examine solutions generated by each scheme. 
We start by numerically estimating the spatial convergence of a 1D density wave with a smooth initial condition. Then, we show plots for the Shu-Osher shock tube problem. 

First, we test the high order spatial convergence of our schemes using a 1D density wave. The initial condition for the 1D compressible Euler density wave is as follows:
\begin{align*}p\left(x,0\right)\aligned{=}1,&v\left(x,0\right)\aligned{=}1.7,&\rho\left(x,0\right)\aligned{=}\frac{1}{2}\sin\left(\pi x\right)+1.\end{align*}
The density wave solution is computed using the \texttt{RK4} timestepper with a fixed $\Delta t=10^{-4}$ to a final time of $T=1$, over the periodic domain $\left[-1,1\right]$. In Tables \eqref{tab:ECAVConvergence} and \eqref{tab:KLConvergence}, the $L^{2}$ norm of the solution at the final time is computed against the analytical solution for various node counts $n$. This is done for polynomial orders $N=2$ through $N=5$. The errors are used to estimate the spatial order of convergence. We observe that the optimal order of convergence $\mathcal{O}\left(\Delta x^{N}\right)$ is attained by each scheme. 
\begin{table}[h!]
\centering
\resizebox{\textwidth}{!}{
\begin{tabulary}{\textwidth}{{|C|C|C|C|C|C|C|C|C|}}
\hline
 & \multicolumn{2}{|c|}{$N=2$} & \multicolumn{2}{|c|}{$N=3$} & \multicolumn{2}{|c|}{$N=4$} & \multicolumn{2}{|c|}{$N=5$} \\
\hline
$n$ & \multicolumn{1}{|c}{$L^{2}\ $Error} & Rate & \multicolumn{1}{|c}{$L^{2}\ $Error} & Rate & \multicolumn{1}{|c}{$L^{2}\ $Error} & Rate & \multicolumn{1}{|c}{$L^{2}\ $Error} & Rate \\
\hline
$16$ & \multicolumn{1}{|c}{$1.65\cdot10^{-1}$} & $-$ & \multicolumn{1}{|c}{$1.59\cdot10^{-2}$} & $-$ & \multicolumn{1}{|c}{$1.59\cdot10^{-2}$} & $-$ & \multicolumn{1}{|c}{$6.43\cdot10^{-3}$} & $-$ \\
\hline
$32$ & \multicolumn{1}{|c}{$4.18\cdot10^{-2}$} & $1.98$ & \multicolumn{1}{|c}{$2.20\cdot10^{-3}$} & $2.86$ & \multicolumn{1}{|c}{$2.20\cdot10^{-3}$} & $2.86$ & \multicolumn{1}{|c}{$4.99\cdot10^{-4}$} & $3.69$ \\
\hline
$64$ & \multicolumn{1}{|c}{$1.05\cdot10^{-2}$} & $2.00$ & \multicolumn{1}{|c}{$2.53\cdot10^{-4}$} & $3.12$ & \multicolumn{1}{|c}{$2.53\cdot10^{-4}$} & $3.12$ & \multicolumn{1}{|c}{$1.87\cdot10^{-5}$} & $4.74$ \\
\hline
$128$ & \multicolumn{1}{|c}{$2.62\cdot10^{-3}$} & $2.00$ & \multicolumn{1}{|c}{$2.11\cdot10^{-5}$} & $3.58$ & \multicolumn{1}{|c}{$2.11\cdot10^{-5}$} & $3.58$ & \multicolumn{1}{|c}{$4.43\cdot10^{-7}$} & $5.40$ \\
\hline
$256$ & \multicolumn{1}{|c}{$6.56\cdot10^{-4}$} & $2.00$ & \multicolumn{1}{|c}{$1.69\cdot10^{-6}$} & $3.64$ & \multicolumn{1}{|c}{$1.69\cdot10^{-6}$} & $3.64$ & \multicolumn{1}{|c}{$9.41\cdot10^{-9}$} & $5.56$ \\
\hline
$512$ & \multicolumn{1}{|c}{$1.64\cdot10^{-4}$} & $2.00$ & \multicolumn{1}{|c}{$1.33\cdot10^{-7}$} & $3.67$ & \multicolumn{1}{|c}{$1.33\cdot10^{-7}$} & $3.67$ & \multicolumn{1}{|c}{$1.78\cdot10^{-10}$} & $5.72$ \\
\hline
\end{tabulary}}
\caption{Errors and estimated orders of convergence of ECAV-FD for degrees$\ N.$}
\label{tab:ECAVConvergence}
\end{table}

\begin{table}[h!]
\centering
\resizebox{\textwidth}{!}{
\begin{tabulary}{\textwidth}{{|C|C|C|C|C|C|C|C|C|}}
\hline
 & \multicolumn{2}{|c|}{$N=2$} & \multicolumn{2}{|c|}{$N=3$} & \multicolumn{2}{|c|}{$N=4$} & \multicolumn{2}{|c|}{$N=5$} \\
\hline
$n$ & \multicolumn{1}{|c}{$L^{2}\ $Error} & Rate & \multicolumn{1}{|c}{$L^{2}\ $Error} & Rate & \multicolumn{1}{|c}{$L^{2}\ $Error} & Rate & \multicolumn{1}{|c}{$L^{2}\ $Error} & Rate \\
\hline
$16$ & \multicolumn{1}{|c}{$1.65\cdot10^{-1}$} & $-$ & \multicolumn{1}{|c}{$1.59\cdot10^{-2}$} & $-$ & \multicolumn{1}{|c}{$1.59\cdot10^{-2}$} & $-$ & \multicolumn{1}{|c}{$6.43\cdot10^{-3}$} & $-$ \\
\hline
$32$ & \multicolumn{1}{|c}{$4.18\cdot10^{-2}$} & $1.98$ & \multicolumn{1}{|c}{$2.20\cdot10^{-3}$} & $2.86$ & \multicolumn{1}{|c}{$2.20\cdot10^{-3}$} & $2.86$ & \multicolumn{1}{|c}{$4.99\cdot10^{-4}$} & $3.69$ \\
\hline
$64$ & \multicolumn{1}{|c}{$1.05\cdot10^{-2}$} & $2.00$ & \multicolumn{1}{|c}{$2.53\cdot10^{-4}$} & $3.12$ & \multicolumn{1}{|c}{$2.53\cdot10^{-4}$} & $3.12$ & \multicolumn{1}{|c}{$1.87\cdot10^{-5}$} & $4.74$ \\
\hline
$128$ & \multicolumn{1}{|c}{$2.62\cdot10^{-3}$} & $2.00$ & \multicolumn{1}{|c}{$2.11\cdot10^{-5}$} & $3.58$ & \multicolumn{1}{|c}{$2.11\cdot10^{-5}$} & $3.58$ & \multicolumn{1}{|c}{$4.43\cdot10^{-7}$} & $5.40$ \\
\hline
$256$ & \multicolumn{1}{|c}{$6.56\cdot10^{-4}$} & $2.00$ & \multicolumn{1}{|c}{$1.69\cdot10^{-6}$} & $3.64$ & \multicolumn{1}{|c}{$1.69\cdot10^{-6}$} & $3.64$ & \multicolumn{1}{|c}{$9.41\cdot10^{-9}$} & $5.56$ \\
\hline
$512$ & \multicolumn{1}{|c}{$1.64\cdot10^{-4}$} & $2.00$ & \multicolumn{1}{|c}{$1.35\cdot10^{-7}$} & $3.64$ & \multicolumn{1}{|c}{$1.35\cdot10^{-7}$} & $3.64$ & \multicolumn{1}{|c}{$1.78\cdot10^{-10}$} & $5.72$ \\
\hline
\end{tabulary}}
\caption{Errors and estimated orders of convergence of KL-FD-HLLC for degrees $N.$}
\label{tab:KLConvergence}
\end{table}

\subsection{Shu-Osher Shock Tube}\label{Shu-Osher Shock Tube}

Next, we examine the solutions generated in the case when some form of stabilization is required to avoid crashing. To do this, we use the Shu-Osher shock tube initial condition for the 1D compressible Euler equations:
\begin{align*}\left(\rho\left(x,0\right),v\left(x,0\right),p\left(x,0\right)\right)=\begin{cases}\left(3.857143,2.629369,10.3333\right)&x<-4\\
\left(1+.2\sin\left(5x\right),0,1\right)&\text{otherwise}.\end{cases}\end{align*}
The Shu-Osher solution is computed using the adaptive \texttt{SSPRK43} timestepper, with the absolute and relative tolerances set to $10^{-6}$ and $10^{-4}$ respectively. The simulation is performed to a final time of $T=1.8$, over the domain $\left[-5,5\right]$, with Dirichlet boundary conditions enforcing that the solution remain constant at the boundaries over time. In Figure \eqref{fig:ShuOsher}, the density of the solutions at the final time computed using $4$th order ECAV-FD and KL-FD-HLLC for $n=500$ and $n=1500$ are shown. These node counts are chosen for consistency with \cite{Popovas25}. Each solution is compared with a reference solution computed using a 5th order WENO method with 25000 nodes. We observe that for $n=500$, KL-FD-HLLC is a closer approximation of the reference solution near the highly oscillatory region to the right. We also observe that to the left side, each finite difference approximation has small oscillations. However, for $n=1500$, we observe that the oscillations are reduced to the left hand side for each finite difference approximation, and we observe that each finite difference solution is a good approximation of the reference solution on the right hand side. We note that high order accurate schemes such as TeCNO \cite{Tecno} obtain solutions for this Shu-Osher shock tube which have fewer oscillations; however, they also introduce upwinding via sign-preserving WENO reconstructions \cite{charles2025learning}. In contrast, the approaches taken in this work are applied to central finite difference schemes; however, the strategies used in this paper could also be applied to TeCNO and WENO schemes \cite{fisher2013high}. 

\begin{figure}%
    \centering
    \subfloat[\centering $n=500\ $nodes]{{\includegraphics[width=0.333\textwidth]{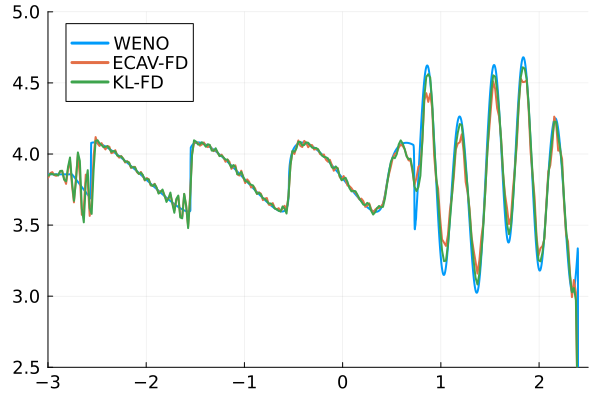} }}%
    \subfloat[\centering $n=1500\ $nodes$,\ $ECAV-FD]{{\includegraphics[width=0.333\textwidth]{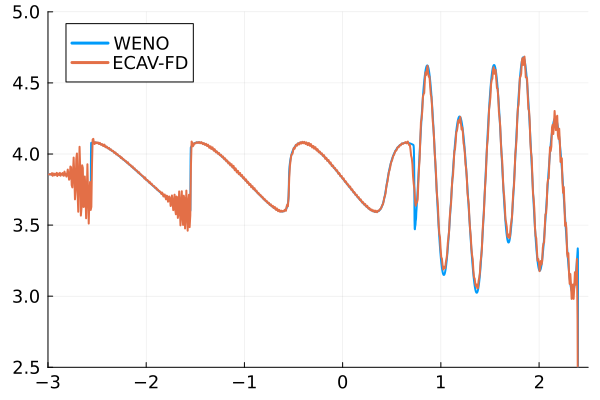} }}%
    \subfloat[\centering $n=1500$ nodes, KL-FD-HLLC]{{\includegraphics[width=0.333\textwidth]{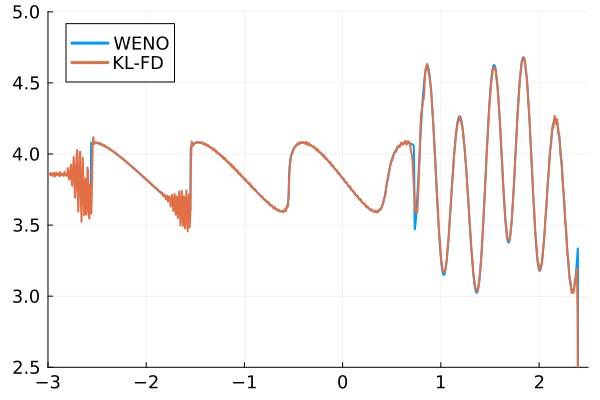} }}%
    \caption{Density of the Shu-Osher shock tube solution at final time$\ T=1.8\ $with$\ N=4.$}%
    \label{fig:ShuOsher}%
\end{figure}

\subsection{Positivity Preservation}\label{Positivity Preservation}

In the simulations performed thus far, preservation of positivity was not strictly necessary. Thus, the limiting coefficients $\fnt{\ell^{c}}$ were set to $\fnt{0}$. In this section, we benchmark KL-FD-HLLC in various difficult benchmarks of positivity preservation. Note that the ECAV-FD scheme is capable of enforcing positivity. In this case, the artificial viscosity must be sufficiently large \cite{Guermond19}. However, for practical purposes, it is often easier to derive the limiting coefficients for the KL-FD schemes. Limiting coefficients are set on a nodewise basis, and then symmetrized, via the following formula, in order to satisfy the relative positivity constraint \eqref{eq:RelativePositivity}:
\begin{align*}\hat{\fnt{\ell^{c}}}_{i}\aligned{=}\begin{cases}1-\frac{\left(1-\alpha\right)}{\Delta t}\left(\frac{\fnt{M}_{ii}\fnt{u}_{i}-\Delta t\cdot\fnt{r}_{i}^{L}}{\fnt{r}_{i}^{H}-\fnt{r}_{i}^{L}}\right)&\fnt{r}_{i}^{H}-\fnt{r}_{i}^{L}>0\\
0&\text{otherwise}\end{cases},&\fnt{\ell^{c}}_{ij}\aligned{=}\max\left\{\hat{\fnt{\ell^{c}}}_{i},\hat{\fnt{\ell^{c}}}_{j}\right\}.\labell{eq:LimitingCoeffs}\end{align*}
The diffusion coefficients are then determined via the strategy described in \eqref{eq:PPKnapsack}. In this section, we test KL-FD-HLLC in various difficult simulations, often used to test positivity preservation methods. 

First, we test KL-FD-HLLC on the 1D compressible Euler, Leblanc shock tube problem. This problem mimics pressure waves in near vacuum conditions, common in hypersonic flow. The initial condition for this problem is as follows:
\begin{align*}p\left(x,0\right)\aligned{=}\begin{cases}10^{9}&x<0\\
1&\text{otherwise}\end{cases},&v\left(x,0\right)\aligned{=}0,&\rho\left(x,0\right)\aligned{=}\begin{cases}2&x<0\\
10^{-3}&\text{otherwise}\end{cases}.\end{align*}
The initial condition has a pressure ratio of $10^{9}$ and a density ratio of $2000$. Due to oscillations formed around shocks in high order simulations which are proportional to the jump, preservation of positivity is necessary to produce a stable solution. The Leblanc shock tube solution is computed with the \texttt{SSPRK43} timestepper using a fixed timestep of $\Delta t=6\cdot10^{-8}$. This timestep is within $10^{-8}$ of the maximum stable timestep for the order $N=6$ simulation. The simulation is performed to a final time of $T=10^{-4}$, over the domain $\left[-10,10\right]$, with Dirichlet boundary conditions, enforcing that the solution remain constant at the boundaries over time. 

In Figure \eqref{fig:Leblanc}, we plot the density, velocity, and pressure of the solution generated by KL-FD-HLLC with $4000$ nodes, against a reference solution. The relative positivity constraint with $\alpha=.5$ is enforced. The reference solution is modified from the solution generated by the discontinuous Galerkin version of quadratic knapsack limiting \cite{Christner25}, with $3$rd order accuracy, and $12000$ elements. The plots are shown for both orders $N=4$ and $N=6$ order accuracy. We observe that the solutions are virtually identical, though there is a slight amount of oscillation present in the velocity of the $N=6$ solution at the peak. However, this tends not to appear with a smaller chosen timestep.

\begin{figure}%
    \centering
    \subfloat[\centering Density,$\ N=4$]{{\includegraphics[width=0.333\textwidth]{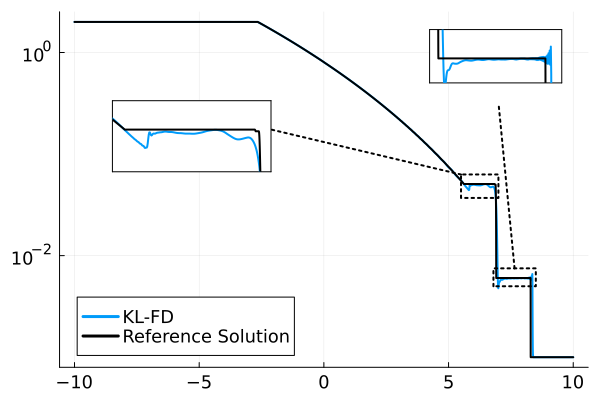} }}%
    \subfloat[\centering Velocity$,\ N=4$]{{\includegraphics[width=0.333\textwidth]{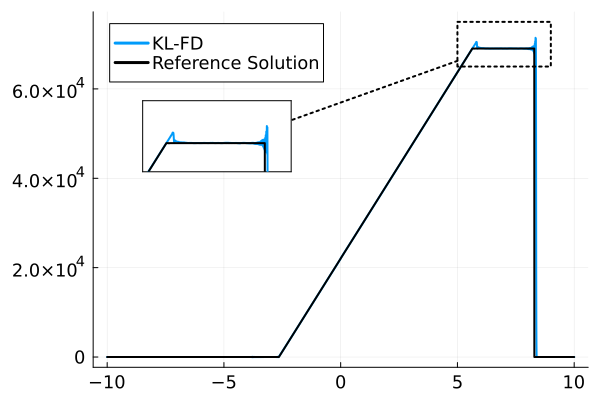} }}%
    \subfloat[\centering Pressure,$\ N=4$]{{\includegraphics[width=0.333\textwidth]{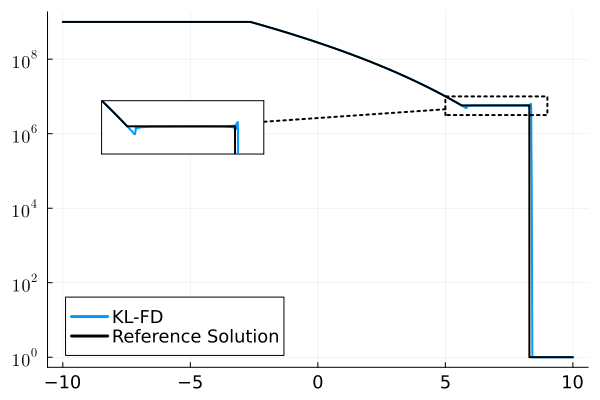} }}\\
    \subfloat[\centering Density,$\ N=6$]{{\includegraphics[width=0.333\textwidth]{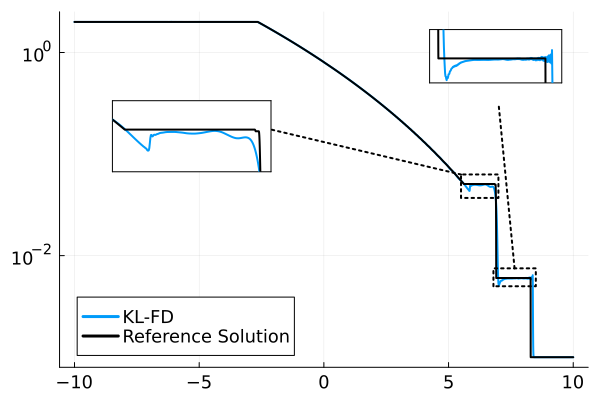} }}%
    \subfloat[\centering Velocity$,\ N=6$]{{\includegraphics[width=0.333\textwidth]{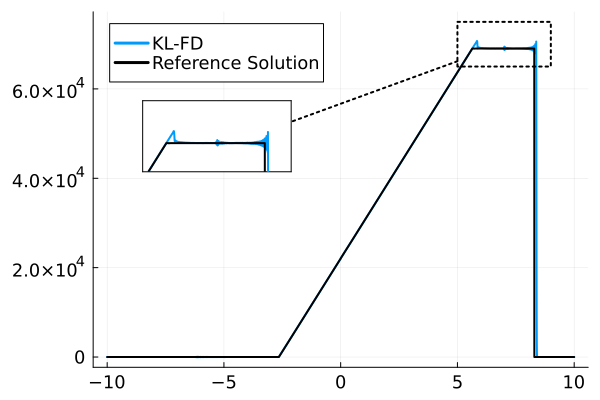} }}%
    \subfloat[\centering Pressure$,\ N=6$]{{\includegraphics[width=0.333\textwidth]{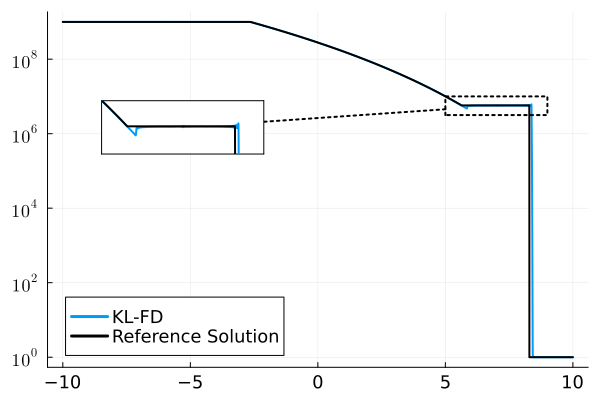} }}%
    \caption{KL-FD-HLLC solution to the Leblanc shock tube with$\ 4000\ $nodes and relative positivity constraint$\ \alpha=.5.$}%
    \label{fig:Leblanc}%
\end{figure}

Next, we test KL-FD-HLLC on the Woodward-Colella blast wave. This problem mimics the interaction of two blast waves, also common in hypersonic flow. The initial condition for this problem is as follows:
\begin{align*}\rho\left(x,0\right)\aligned{=}1,&v\left(x,0\right)\aligned{=}0,&p\left(x,0\right)\aligned{=}\begin{cases}10^{3}&x<.1\\
10^{-2}&.1\le x<.9\\
10^{2}&x\ge.9\end{cases}.\end{align*}
The initial condition has a pressure ratio of $10^{5}$ at the left interface, and $10^{4}$ at the right interface, agian making preservation of positivity necessary for a stable simulation. The Woodward-Colella blast wave solution is computed using the \texttt{SSPRK43} timestepper with a fixed timestep of $\Delta t=2\cdot10^{-5}$. This timestep is within $10^{-5}$ of the maximum stable timestep for the order $N=6$ simulation. The simulation is performed to a final time of $T=.038$, over the domain $\left[0,1\right]$, with reflective boundary conditions.

In Figure \eqref{fig:Woodward}, we plot the density of the solution generated by KL-FD-HLLC with the HLLC flux with $1200$ nodes, against a reference solution. The relative positivity constraint with $\alpha=.1$ is enforced. The reference solution is modified from the solution generated by the low order method \eqref{eq:LOM} with a $4$th order stencil and $25000$ nodes. The plots are shown for both orders $N=4$ and $N=6$ order accuracy. We observe that the solutions are similar, but the order $N=6$ simulation is a bit more dissipative. This is likely partially due to the fact that the limiting coefficients, as obtained in \eqref{eq:LimitingCoeffs}, must be made larger to account for a wider stencil in the symmetrization process.

\begin{figure}%
    \centering
    \subfloat[\centering $N=4$]{{\includegraphics[width=0.500\textwidth]{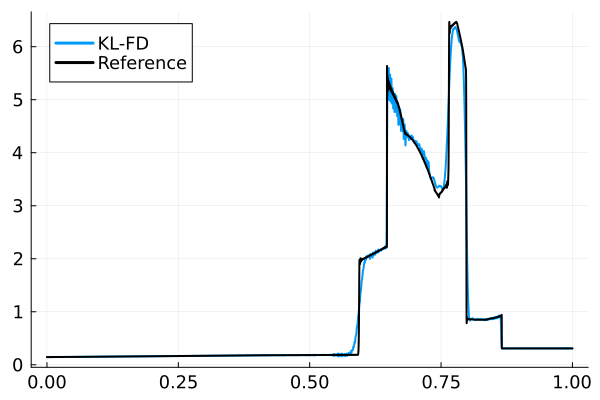} }}%
    \subfloat[\centering $N=6$]{{\includegraphics[width=0.500\textwidth]{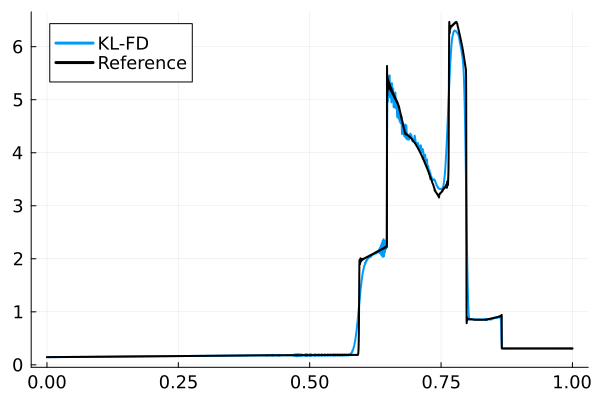} }}%
    \caption{KL-FD-HLLC solution to the Woodward-Colella blast wave with 1200 nodes and relative positivity constraint$\ \alpha=.1.$}%
    \label{fig:Woodward}%
\end{figure}

\subsection{Two-Dimensional Solutions}\label{Two-Dimensional Solutions}

In this section, we test ECAV-FD and KL-FD-HLLC in two dimensional simulations. In particular, we examine the density arising from the Kelvin-Helmholtz instability problem (KHI). This problem is often used to test the behavior of 2D numerical methods in the presence of turbulent-like flow. Without enforcing positivity ($\fnt{\ell^c} = 0$), collocation-type entropy stable methods tend to crash with negative density for KHI around $T\approx3.5$ \cite{KHIPos}. However, both ECAV-FD and KL-FD-HLLC appear to be able to simulate KHI until and past $T=25$ without enforcing positivity. In this case, we call this \textit{long-time} Kelvin-Helmholtz instability. The initial condition for KHI is:
\begin{align*}\rho\left(x,y,0\right)\aligned{=}\frac{1}{2}+\frac{3}{4}B,&v_{1}\left(x,y,0\right)\aligned{=}\frac{1}{2}\left(B-1\right),\\
v_{2}\left(x,y,0\right)\aligned{=}\frac{1}{10}\sin\left(2\pi x\right),&p\left(x,y,0\right)\aligned{=}1,\end{align*}
where $B=\tanh\left(15y+7.5\right)-\tanh\left(15y-7.5\right)$. The long-time KHI solution is computed on the periodic domain $\left[-1,1\right]^{2}$ to a final time $T=25$, using the adaptive \texttt{SSPRK43} timestepper, with an absolute tolerance of $10^{-6}$ and relative tolerance of $10^{-4}$. In Figures \eqref{fig:KHI} and \eqref{fig:KHIKL}, we plot the solutions generated by ECAV-FD and KL-FD-HLLC respectively, at times $t=5$, $t=10$, and $t=25$, each for orders $N=4$ and $N=6$, and each with $512$ nodes along each dimension. The color range was truncated to $\left[.25,2.5\right]$. The resulting solutions appear to exhibit nonphysical oscillations, which can be mitigated by introducing shock capturing methods.

\begin{figure}%
    \centering
    \subfloat[\centering $N=4,\ t=5$]{{\includegraphics[width=0.333\textwidth]{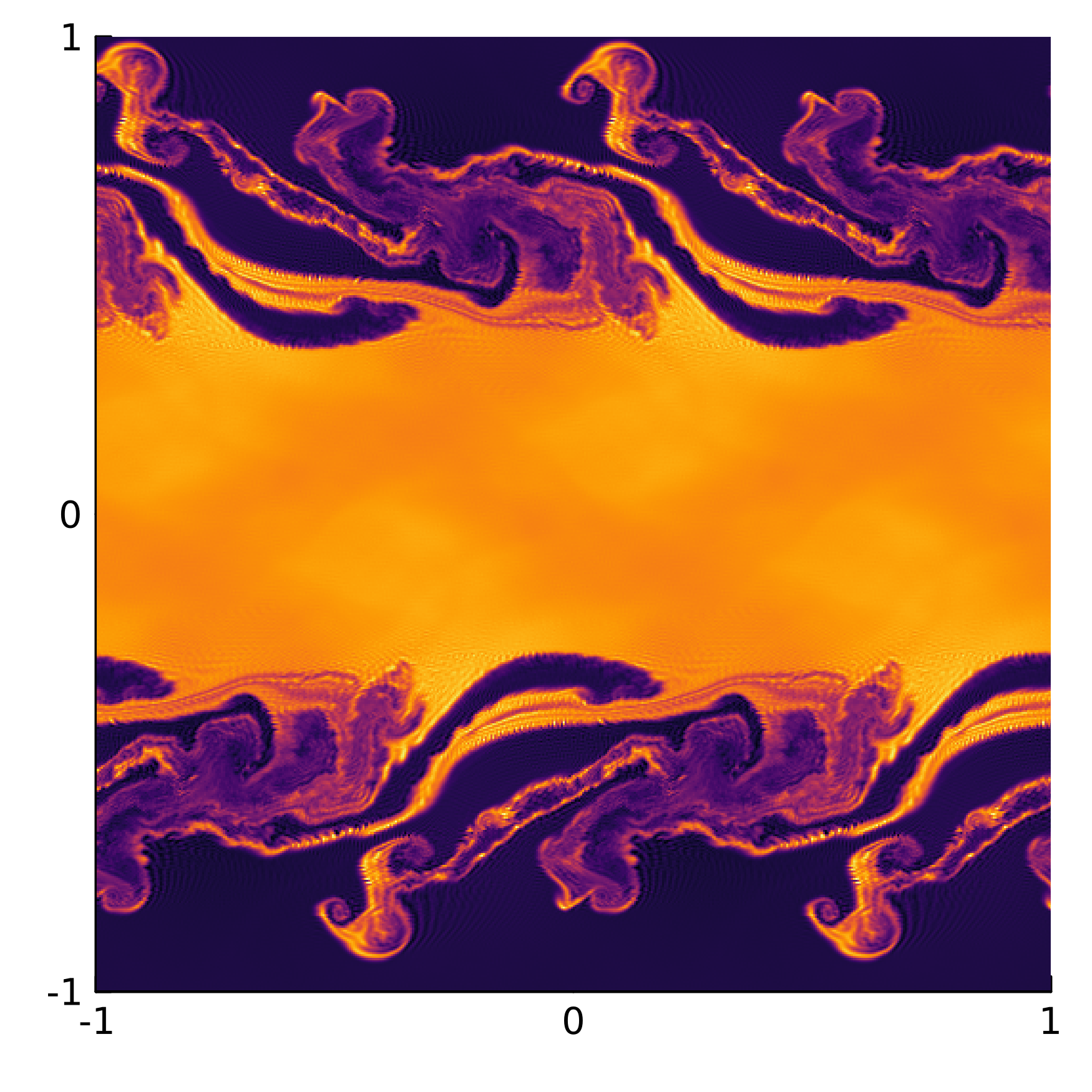} }}%
    \subfloat[\centering $N=4,\ t=10$]{{\includegraphics[width=0.333\textwidth]{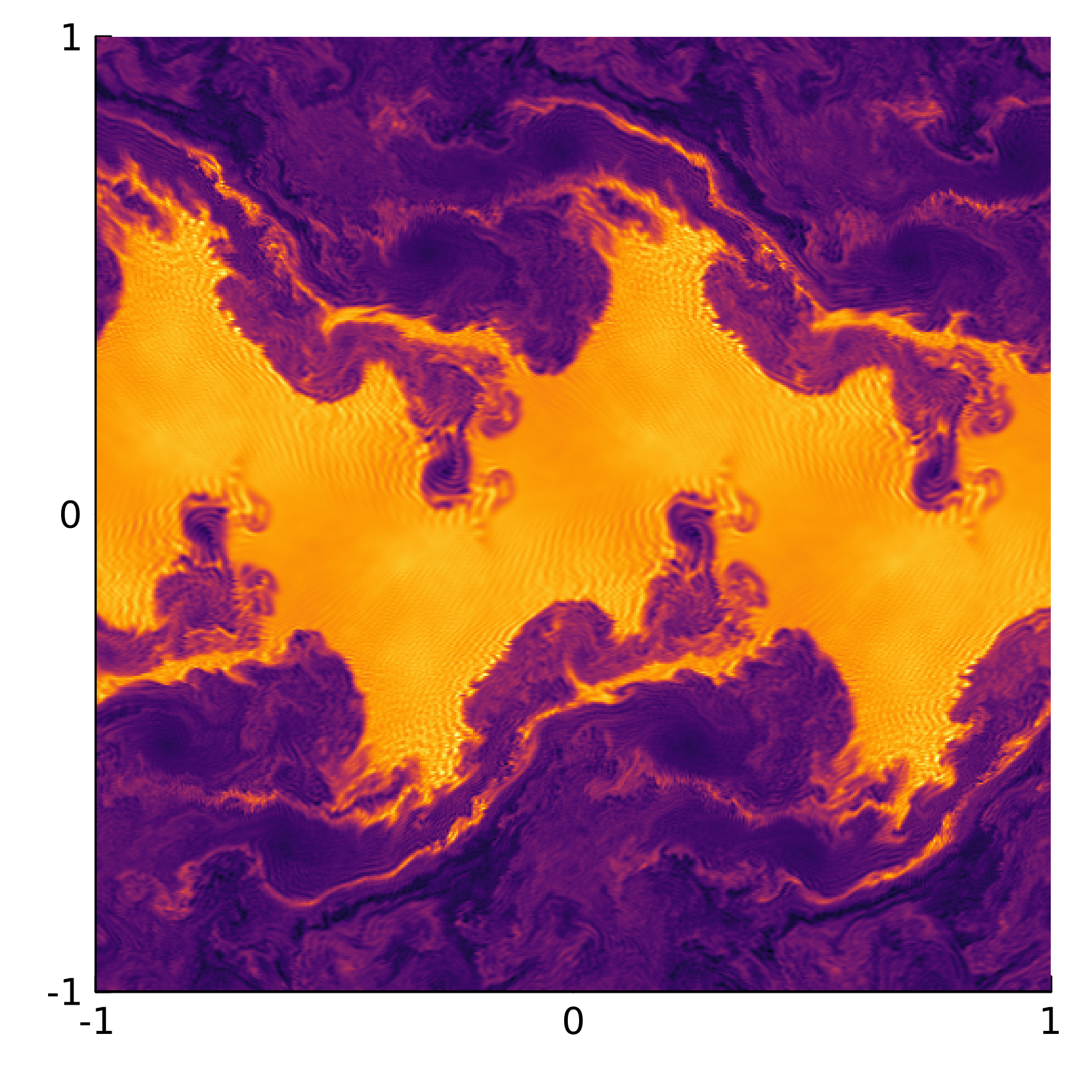} }}%
    \subfloat[\centering $N=6,\ t=25$]{{\includegraphics[width=0.333\textwidth]{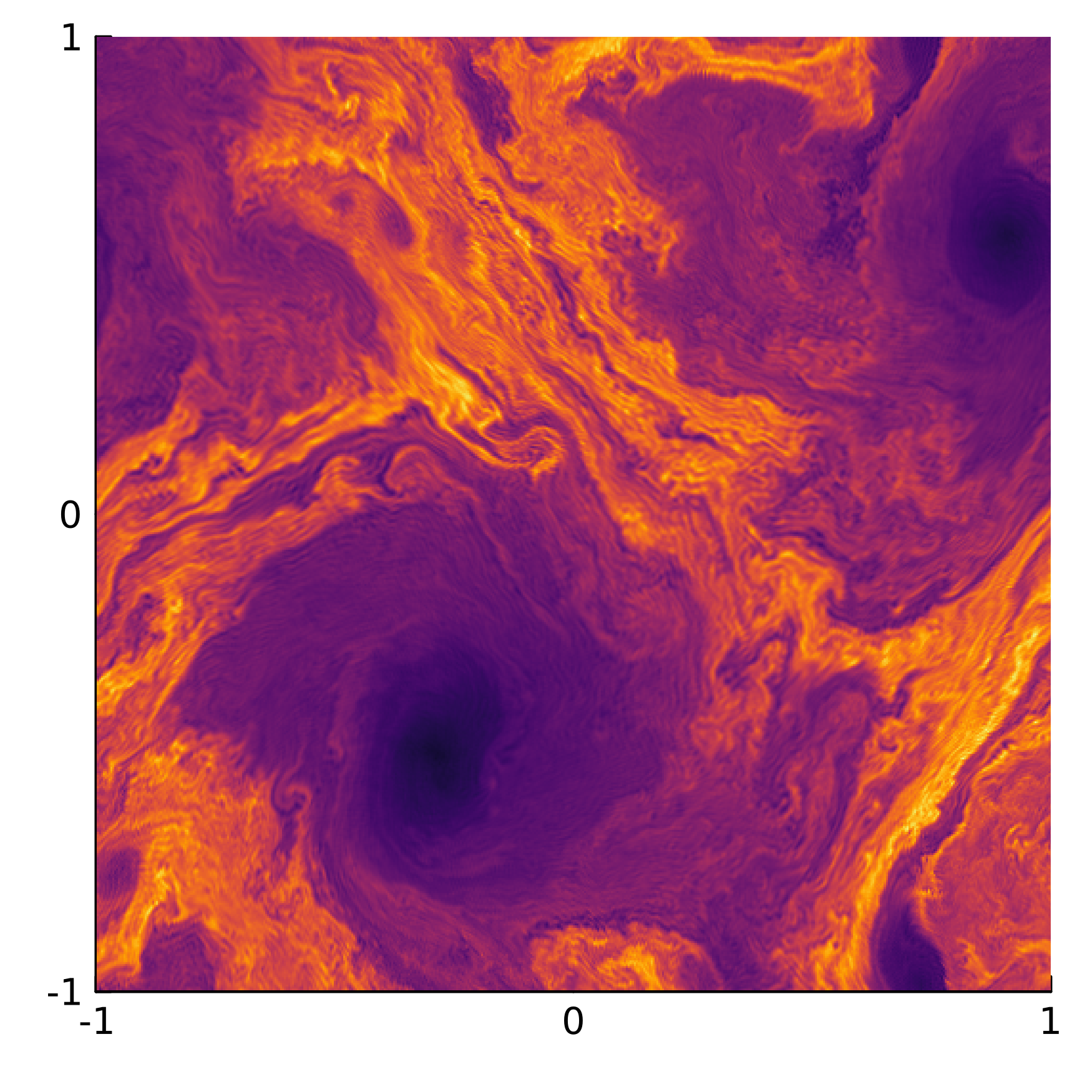} }}\\
    \subfloat[\centering $N=6,\ t=5$]{{\includegraphics[width=0.333\textwidth]{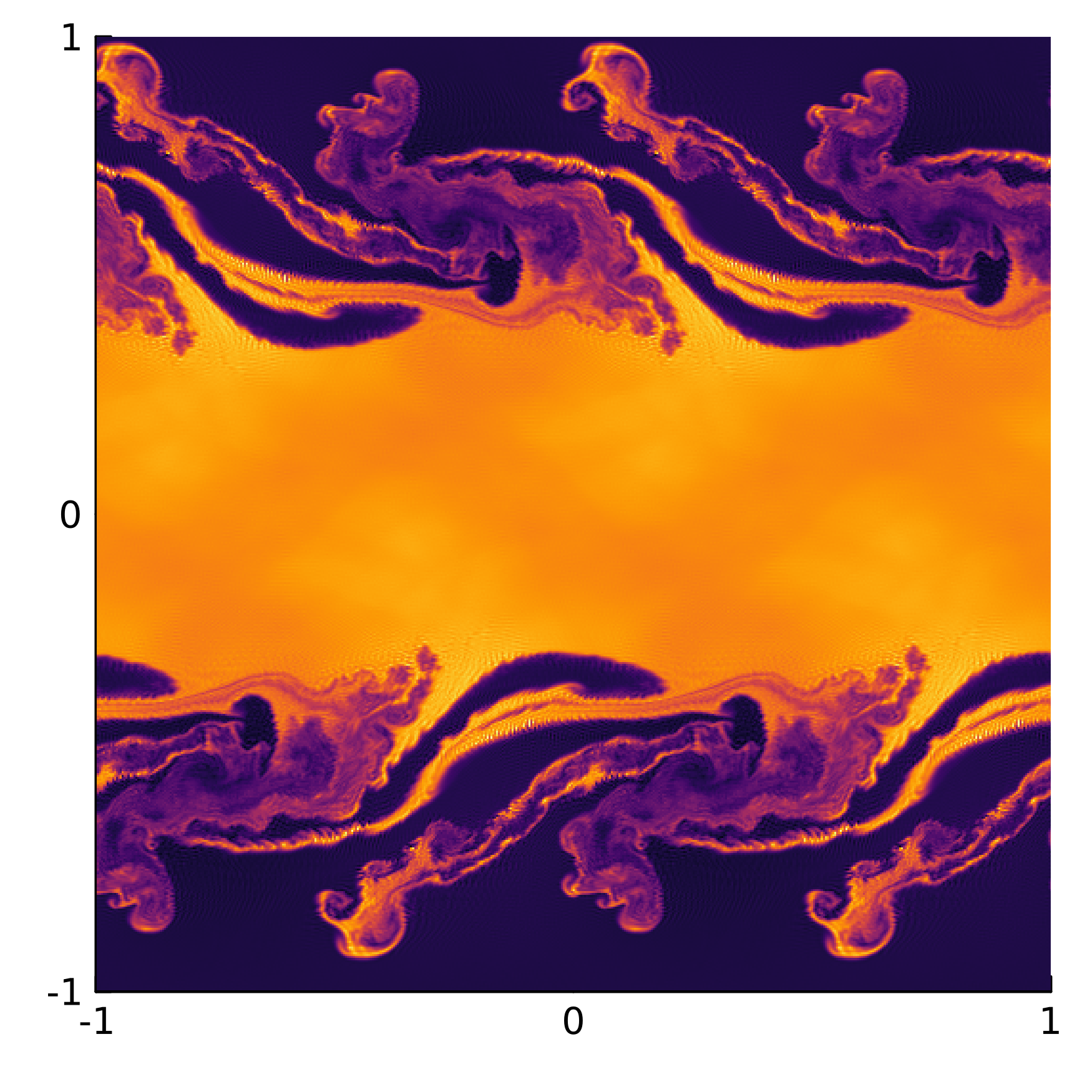} }}%
    \subfloat[\centering $N=6,\ t=10$]{{\includegraphics[width=0.333\textwidth]{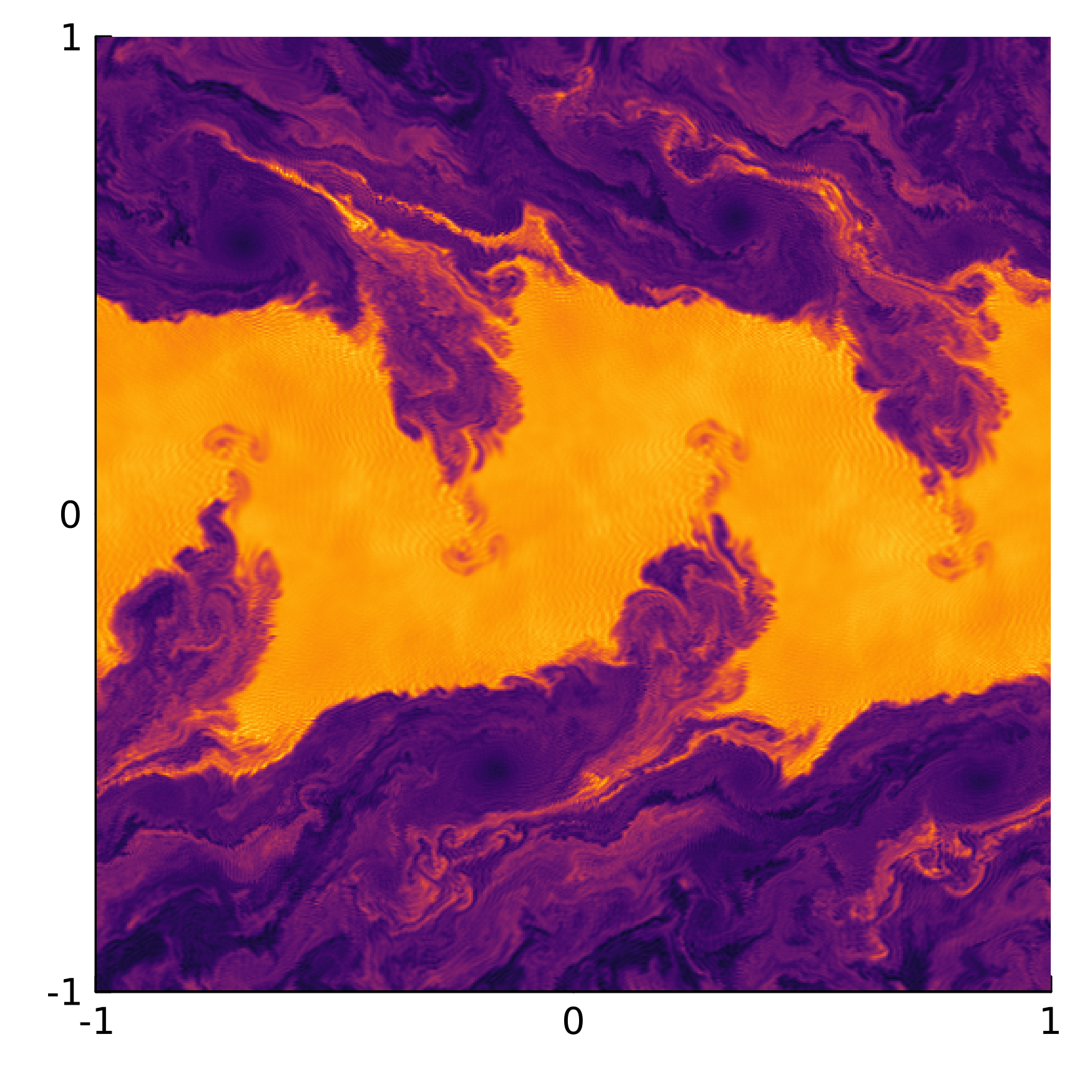} }}%
    \subfloat[\centering $N=6,\ t=25$]{{\includegraphics[width=0.333\textwidth]{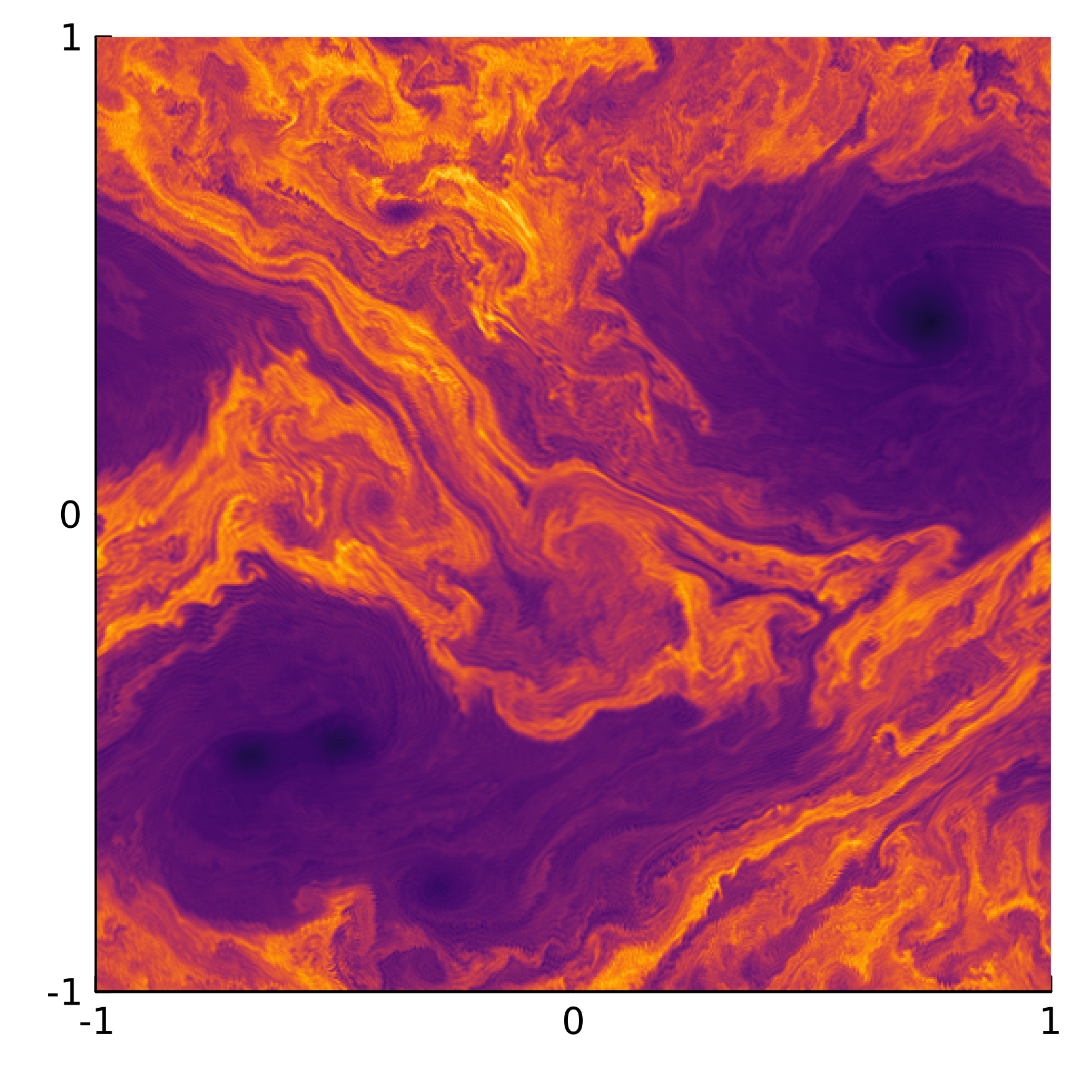} }}%
    \caption{ECAV-FD solution to Kelvin-Helmholtz instability problem.}%
    \label{fig:KHI}%
\end{figure}

\begin{figure}%
    \centering
    \subfloat[\centering $N=4,\ t=5$]{{\includegraphics[width=0.333\textwidth]{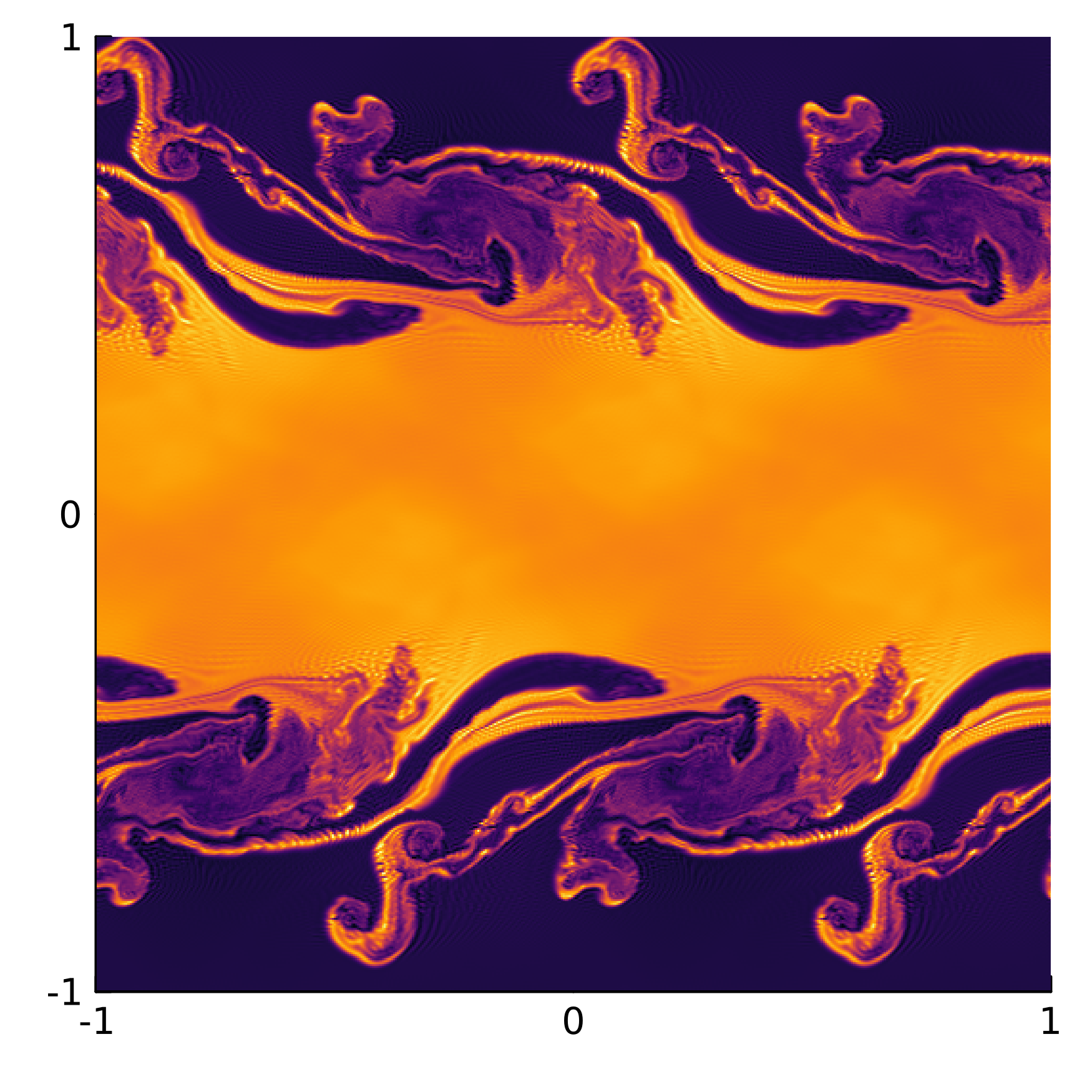} }}%
    \subfloat[\centering $N=4,\ t=10$]{{\includegraphics[width=0.333\textwidth]{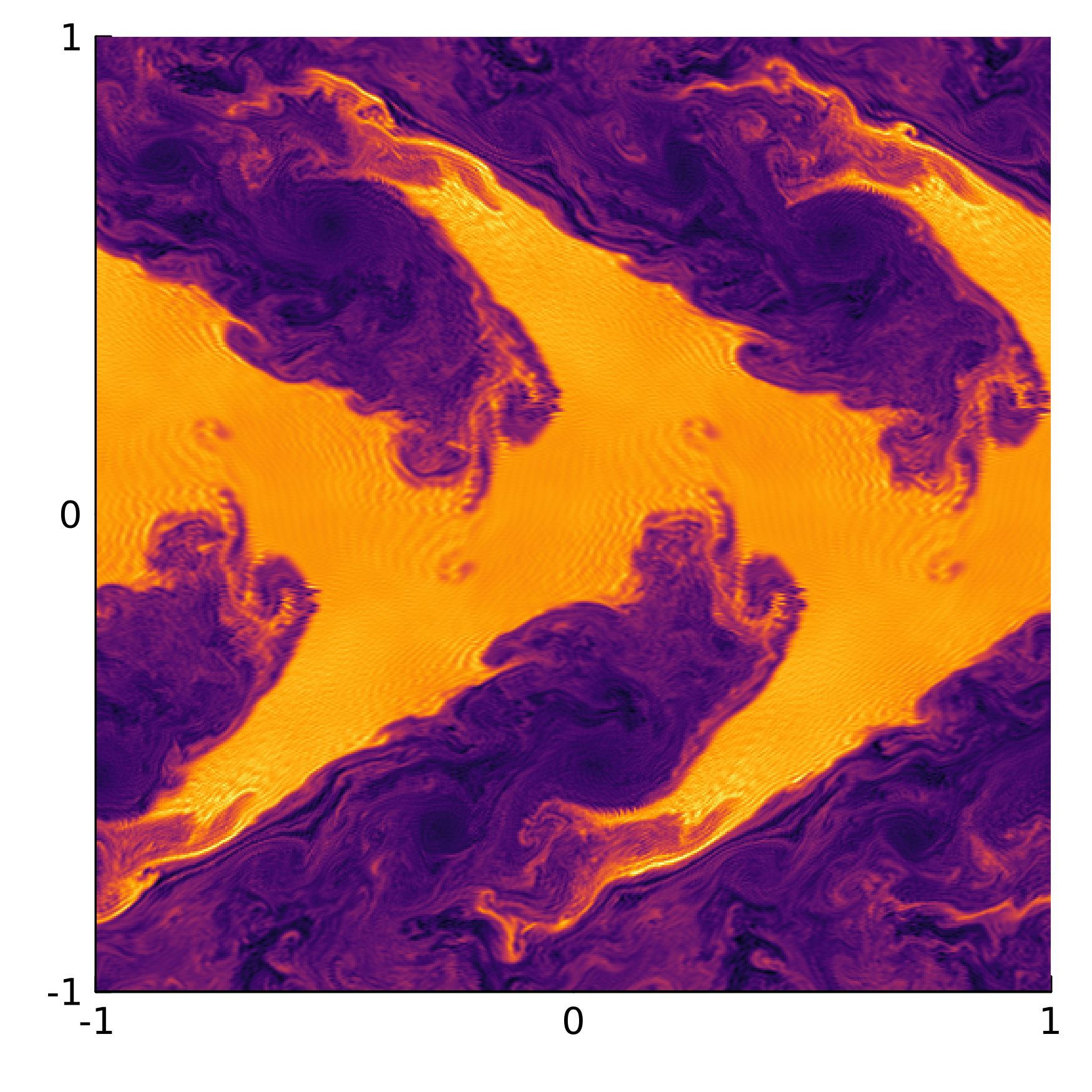} }}%
    \subfloat[\centering $N=4,\ t=25$]{{\includegraphics[width=0.333\textwidth]{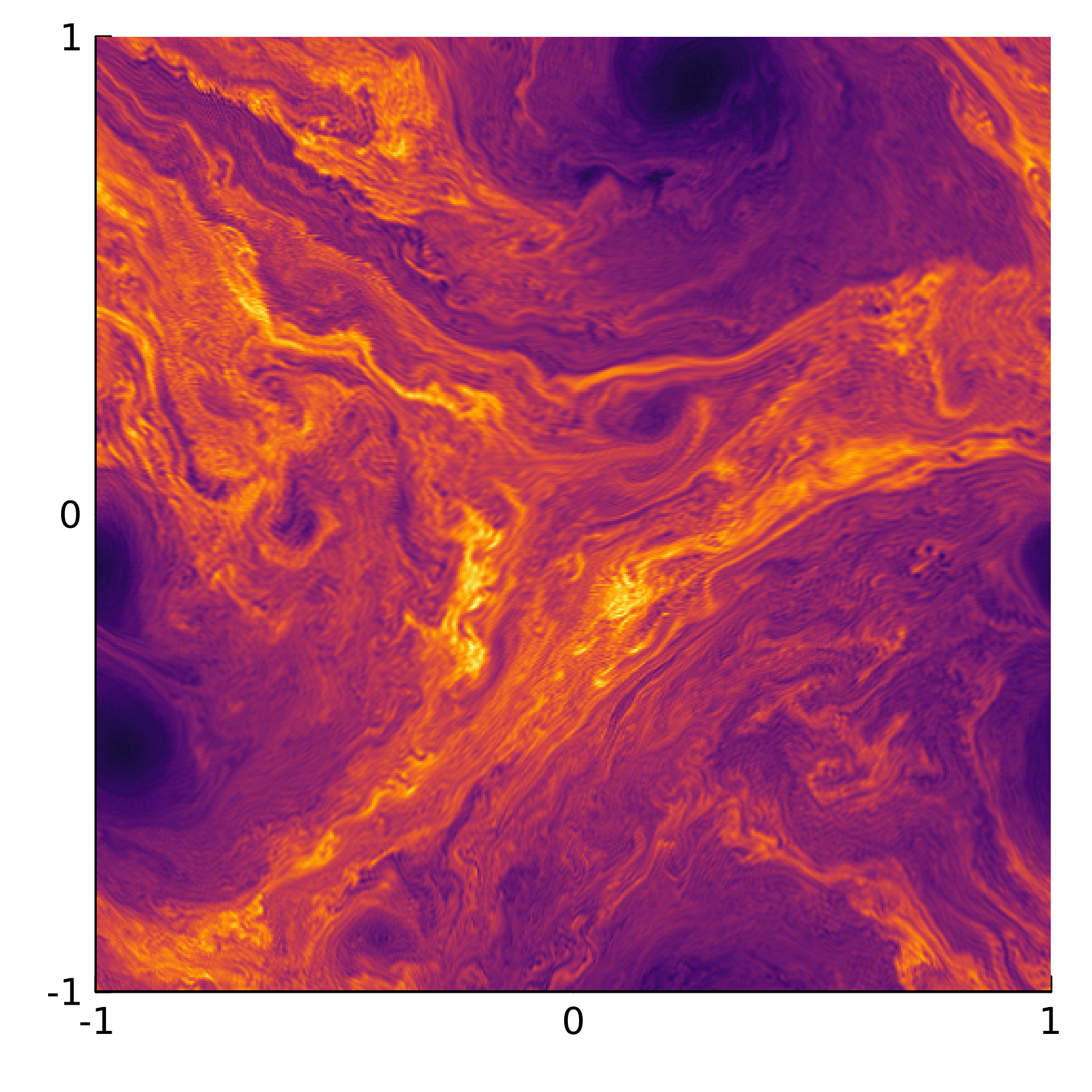} }}\\
    \subfloat[\centering $N=6,\ t=5$]{{\includegraphics[width=0.333\textwidth]{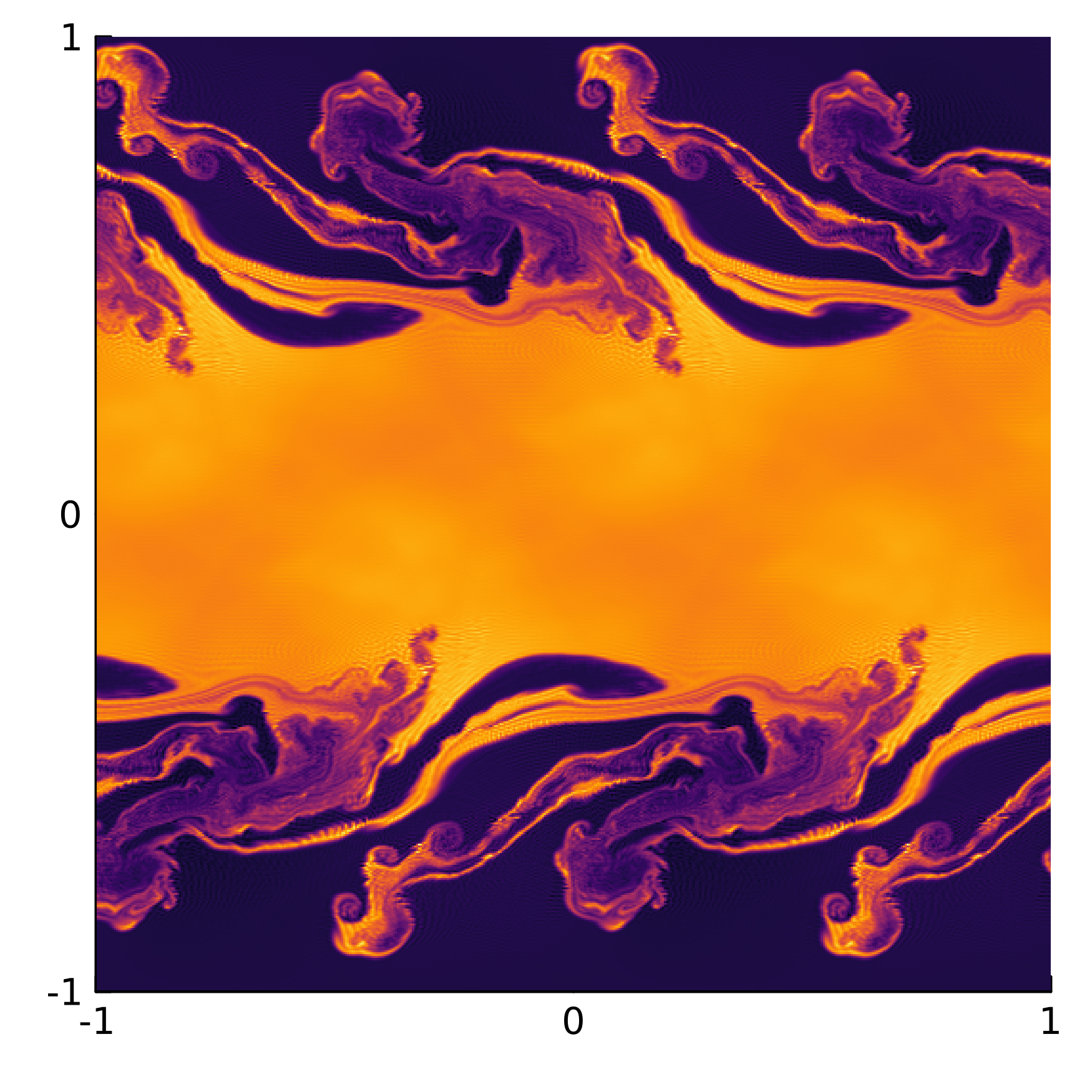} }}%
    \subfloat[\centering $N=6,\ t=10$]{{\includegraphics[width=0.333\textwidth]{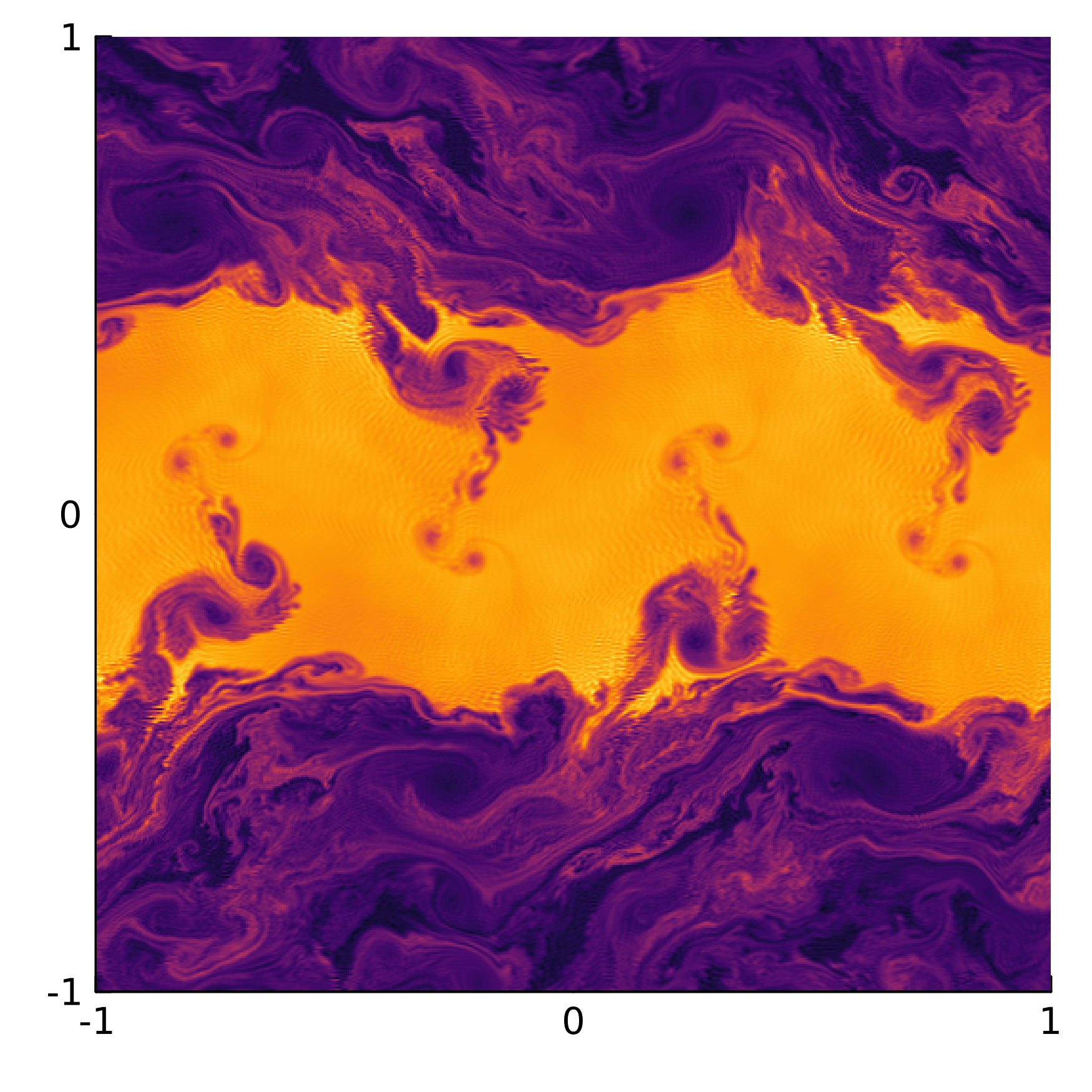} }}%
    \subfloat[\centering $N=6,\ t=25$]{{\includegraphics[width=0.333\textwidth]{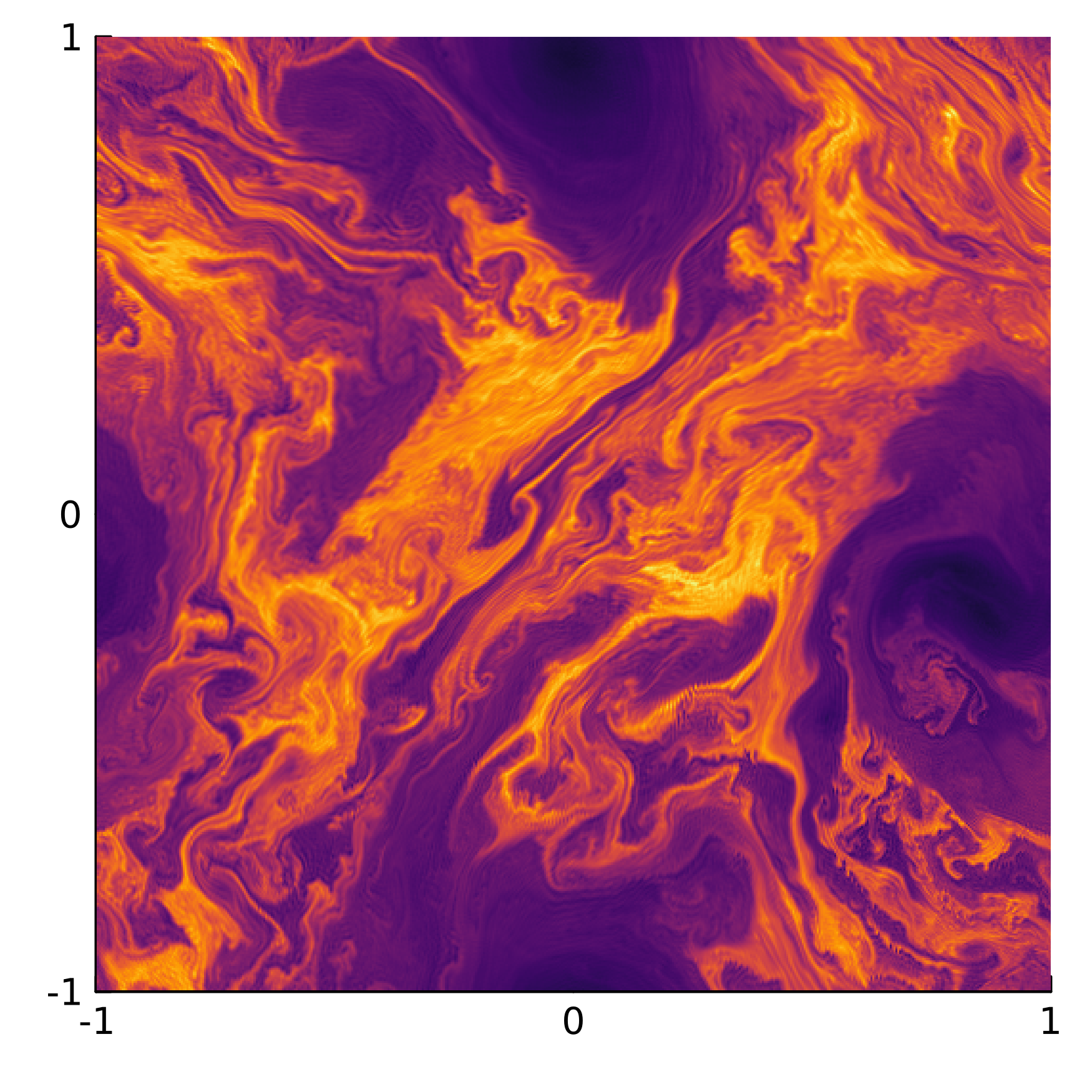} }}%
    \caption{KL-FD-HLLC solution to Kelvin-Helmholtz instability problem.}%
    \label{fig:KHIKL}%
\end{figure}
\section{Conclusion}\label{Conclusion}

In this work, we proposed two approaches for constructing high order accurate, conservative, entropy stable finite difference methods. The first is an entropy correction artificial viscosity method, which introduces a minimal amount of artificial viscosity in order to satisfy a semi-discrete entropy inequality. The second is a knapsack limiting technique, which minimally blends a low order method into a high order method in order to satisfy the same semi-discrete entropy inequality. We also discussed how knapsack limiting can be expressed as a weighted form of the entropy correction artificial viscosity method. Furthermore, knapsack limiting schemes can be used to preserve positivity. We also discussed how weakly-enforced boundary conditions can be imposed using summation by parts finite difference operators. Each scheme exhibits high order accuracy for sufficiently regular solutions, while also retaining entropy stability and conservation. 

\section{Acknowledgements}\label{Acknowledgements}

The authors gratefully acknowledge helpful discussions with Ayaboe Edoh and Brad Maeng, as well as support from National Science Foundation awards DMS-1943186 and DMS-223148. 

\section{Appendix: Optimization-Based Relaxation of the Entropy Inequality}\label{Optimization-Based Relaxation of the Entropy Inequality}

In this section, we will discuss a strategy of decoupling the diffusion coefficients in the spatial derivative of the flux and the spatial derivative of the entropy flux. The strategy allows for the artificial diffusion of the spatial derivative of the entropy flux to be decoupled from the diffusion applied to $\fnt{r}$ from the diffusion coefficients. In effect, this results in a relaxation of the entropy inequality \eqref{eq:DiscreteEI}. It does this by applying only a slight increase in the dimension of the optimization problems \eqref{eq:Knapsack} and \eqref{eq:PPKnapsack}, which does not scale with spatial dimension. 

We will demonstrate in the corresponding numerical results section that the ECAV-FD and KL-FD schemes result in excessive dissipation in high orders, and the relaxation discussed in this section helps to recover it. However, we will also show that relaxing the entropy inequality results in more oscillations in the generated solutions. Furthermore, relaxing the entropy inequality too much may result in an unstable scheme. The resulting strategy will be derived for the entropy correction artificial viscosity scheme, but the strategy for the knapsack limiting scheme will also be shown. Recall, the spatial derivative of the entropy flux is discretized as follows:
\begin{align*}\left[\sum_{k=1}^{d}\frac{\partial F_{k}\left(\vecf{u}\right)}{\partial x_{k}}\right]_{x_{i}}\aligned{\approx}\frac{1}{\fnt{M}_{ii}}\sum_{j}^{ }\left\lVert\fnt{n}_{ij}\right\rVert\left(\fnt{v}_{j}^{T}\vecf{f}_{ij}\left(\fnt{\theta}_{ij}\right)-\left(\psi\left(\fnt{u}_{j}\right)-\psi\left(\fnt{u}_{i}\right)\right)\cdot\hat{\fnt{n}}_{ij}\right).\end{align*}
The diffusion coefficients $\fnt{\theta}_{ij}$, which apply diffusion within the flux $\vecf{f}_{ij}\left(\fnt{\theta}_{ij}\right)$ to determine $\fnt{r}$ as in \eqref{eq:BlendedScheme}, also apply diffusion to the spatial derivative of the entropy flux. However, this does not need to be the case. Rather, we can instead replace $\fnt{\theta}_{ij}$ in the discretization of the spatial derivative of the entropy flux with an additional free variable, as such:
\begin{align*}\left[\sum_{k=1}^{d}\frac{\partial F_{k}\left(\vecf{u}\right)}{\partial x_{k}}\right]_{x_{i}}\aligned{\approx}\frac{1}{\fnt{M}_{ii}}\sum_{j}^{ }\left\lVert\fnt{n}_{ij}\right\rVert\left(\fnt{v}_{j}^{T}\vecf{f}_{ij}\left(\fnt{\tau}_{i}\right)-\left(\psi\left(\fnt{u}_{j}\right)-\psi\left(\fnt{u}_{i}\right)\right)\cdot\hat{\fnt{n}}_{ij}\right),\end{align*}
where $\fnt{\tau}\in\mathbb{R}^{n}$. At this point, the discretization of the entropy inequality \eqref{eq:EntropyInequality} has a new form:
\begin{align*}\fnt{v}_{i}^{T}\frac{\text{d}\fnt{u}_{i}}{\text{d}t}+\frac{1}{\fnt{M}_{ii}}\sum_{j}^{ }\left\lVert\fnt{n}_{ij}\right\rVert\left(\fnt{v}_{j}^{T}\vecf{f}_{ij}\left(\fnt{\tau}_{i}\right)-\left(\psi\left(\fnt{u}_{j}\right)-\psi\left(\fnt{u}_{i}\right)\right)\cdot\hat{\fnt{n}}_{ij}\right)\aligned{\le}0,\\
\fnt{v}_{i}^{T}\left(-\frac{1}{\fnt{M}_{ii}}\sum_{j}^{ }\left\lVert\fnt{n}_{ij}\right\rVert\vecf{f}_{ij}\left(\fnt{\theta}_{ij}\right)\right)+\frac{1}{\fnt{M}_{ii}}\sum_{j}^{ }\left\lVert\fnt{n}_{ij}\right\rVert\left(\fnt{v}_{j}^{T}\vecf{f}_{ij}\left(\fnt{\tau}_{i}\right)-\left(\psi\left(\fnt{u}_{j}\right)-\psi\left(\fnt{u}_{i}\right)\right)\cdot\hat{\fnt{n}}_{ij}\right)\aligned{\le}0,\tag{$\text{by}\ \eqref{eq:BlendedScheme}$}\\
\sum_{j}^{ }\left\lVert\fnt{n}_{ij}\right\rVert\left[\left(\fnt{v}_{j}^{T}\vecf{f}_{ij}\left(\fnt{\tau}_{i}\right)-\fnt{v}_{i}^{T}\vecf{f}_{ij}\left(\fnt{\theta}_{ij}\right)\right)-\left(\psi\left(\fnt{u}_{j}\right)-\psi\left(\fnt{u}_{i}\right)\right)\cdot\hat{\fnt{n}}_{ij}\right]\aligned{\le}0,\tag{$\fnt{M}_{ii}>0$}\\
-\sum_{j}^{ }\left\lVert\fnt{n}_{ij}\right\rVert\left[\left(\fnt{v}_{j}-\fnt{v}_{i}\right)^{T}\vecf{f}_{ij}^{H}-\left(\psi\left(\fnt{u}_{j}\right)-\psi\left(\fnt{u}_{i}\right)\right)\cdot\hat{\fnt{n}}_{ij}\right]\aligned{ }\\
+\sum_{j}^{ }\left\lVert\fnt{n}_{ij}\right\rVert\left[\fnt{v}_{i}^{T}\left(\fnt{u}_{i}-\fnt{u}_{j}\right)\right]\fnt{\theta}_{ij}-\sum_{j}^{ }\left\lVert\fnt{n}_{ij}\right\rVert\left[\fnt{v}_{j}^{T}\left(\fnt{u}_{i}-\fnt{u}_{j}\right)\right]\fnt{\tau}_{i}\aligned{\ge}0,\\
\fnt{a}_{i}^{T}\begin{bmatrix}\fnt{\theta}_{i}\\
\fnt{\tau}_{i}\end{bmatrix}\aligned{\ge}b_{i},\end{align*}
where
\begin{align*}\left(\fnt{a}_{i}\right)_{j}\aligned{=}\fnt{a}_{ij}=\begin{cases}\left\lVert\fnt{n}_{ij}\right\rVert\left[\fnt{v}_{i}^{T}\left(\fnt{u}_{i}-\fnt{u}_{j}\right)\right]&j=\text{end}\\
-\sum_{j}^{ }\left\lVert\fnt{n}_{ij}\right\rVert\left[\fnt{v}_{j}^{T}\left(\fnt{u}_{i}-\fnt{u}_{j}\right)\right]&\text{otherwise}\end{cases},\\
b_{i}\aligned{=}\sum_{j}^{ }\left\lVert\fnt{n}_{ij}\right\rVert\left[\left(\fnt{v}_{j}-\fnt{v}_{i}\right)^{T}\vecf{f}_{ij}^{H}-\left(\psi\left(\fnt{u}_{j}\right)-\psi\left(\fnt{u}_{i}\right)\right)\cdot\hat{\fnt{n}}_{ij}\right].\end{align*}
We call this the \textit{relaxed entropy inequality}. Then, the pre-symmetrized diffusion coefficients and $\fnt{\tau}$-coefficients can be determined through the same optimization problem as before:
\begin{align*}\min_{\begin{matrix}\fnt{a}_{i}^{T}\begin{bmatrix}\hat{\fnt{\theta}}\\
\fnt{\tau}\end{bmatrix}\ge b_{i}\\
\hat{\fnt{\theta}}\ge0\\
\fnt{\tau}\ge0\end{matrix}}\begin{bmatrix}\hat{\fnt{\theta}}\\
\fnt{\tau}\end{bmatrix}^{T}\begin{bmatrix}\hat{\fnt{\theta}}\\
\fnt{\tau}\end{bmatrix}.\labell{eq:DKnapsack}\end{align*}
We note that the property $\fnt{a}_{i}\ge0$ is lost in this strategy. However, the above problem still yields an explicit solution. Note, it can be rewritten with elementwise upper bounds. Specifically, it is equivalent when given the constraint $0\le\hat{\fnt{\theta}}\le\infty$. It is shown in \cite{Christner25} that Algorithm 3.1 converges in exactly one iteration to the solution for such a problem. Therefore, an explicit formula can be derived for the optimal solution for the optimization problem \eqref{eq:DKnapsack}:
\begin{align*}\fnt{a}_{ij}^{c}\aligned{=}\begin{cases}\fnt{a}_{ij}&\fnt{a}_{ij}\ge0\\
0&\text{otherwise}\end{cases},&\begin{bmatrix}\hat{\fnt{\theta}}_{i}\\
\fnt{\tau}_{i}\end{bmatrix}=\frac{b_{i}\fnt{a}_{i}^{c}}{\fnt{a}_{i}^{T}\fnt{a}_{i}^{c}}.\end{align*}
Note that since each $\fnt{\tau}$-coefficient is unique to each node and does not apply diffusion to $\fnt{r}$, no symmetrization is needed for the $\fnt{\tau}$-coefficients: just for the diffusion coefficients. 

This strategy of relaxing the entropy inequality can also be applied to knapsack limiting. In which case, the resulting vector $\fnt{a}_{i}$ and scalar $b_{i}$ have the following form:
\begin{align*}\left(\fnt{a}_{i}\right)_{j}\aligned{=}\fnt{a}_{ij}=\begin{cases}\left\lVert\fnt{n}_{ij}\right\rVert\left[\fnt{v}_{i}^{T}\left(\vecf{f}_{ij}^{L}-\vecf{f}_{ij}^{H}\right)\right]&j=\text{end}\\
-\sum_{j}^{ }\left\lVert\fnt{n}_{ij}\right\rVert\left[\fnt{v}_{j}^{T}\left(\vecf{f}_{ij}^{L}-\vecf{f}_{ij}^{H}\right)\right]&\text{otherwise}\end{cases},\\
b_{i}\aligned{=}\sum_{j}^{ }\left\lVert\fnt{n}_{ij}\right\rVert\left[\left(\fnt{v}_{j}-\fnt{v}_{i}\right)^{T}\vecf{f}_{ij}^{H}-\left(\psi\left(\fnt{u}_{j}\right)-\psi\left(\fnt{u}_{i}\right)\right)\cdot\hat{\fnt{n}}_{ij}\right].\end{align*}
Then, the pre-symmetrization diffusion coefficients can be determined through the same quadratic continuous knapsack problem:
\begin{align*}\min_{\begin{matrix}\fnt{a}_{i}^{T}\begin{bmatrix}\hat{\fnt{\theta}}\\
\fnt{\tau}\end{bmatrix}\ge b_{i}\\
0\le\hat{\fnt{\theta}}\le1-\fnt{\ell^{c}}_{i}\\
0\le\fnt{\tau}\le1\end{matrix}}\begin{bmatrix}\hat{\fnt{\theta}}\\
\fnt{\tau}\end{bmatrix}^{T}\begin{bmatrix}\hat{\fnt{\theta}}\\
\fnt{\tau}\end{bmatrix}.\labell{eq:PPKnapsack}\end{align*}
\subsection{Relaxed Entropy Inequality Numerical Results}\label{Relaxed Entropy Inequality Numerical Results}

In this section, we revisit the numerical experiments done in Section \eqref{Numerical Results} with the relaxed versions of ECAV-FD and KL-FD-HLLC. We will refer to each scheme as RECAV-FD and RKL-FD-HLLC respectively. We will demonstrate that the ECAV-FD and KL-FD schemes result in excessive dissipation in high orders, while RECAV-FD and RKL-FD-HLLC help to recover it. However, we will also show that RECAV-FD and RKL-FD-HLLC result in more oscillations in the generated solutions. 

First, we perform the same convergence experiments to analyze the order of spatial convergence of the relaxed schemes. In Tables \eqref{tab:DECAVConvergence} and \eqref{tab:DKLConvergence}, the $L^{2}$ norm of the solution at the final time is computed against the analytical solution for various node counts $n$. This is done for polynomial orders $N=2$ through $N=5$. The errors are again used to estimate the spatial order of convergence. We observe that the optimal order of convergence $\mathcal{O}\left(\Delta x^{N}\right)$ is attained by each scheme. Moreover, we observe that for odd $N$, the errors exhibits an improved order $\mathcal{O}\left(\Delta x^{N+1}\right)$ convergence. 

\begin{table}[h!]
\centering
\resizebox{\textwidth}{!}{
\begin{tabulary}{\textwidth}{{|C|C|C|C|C|C|C|C|C|}}
\hline
 & \multicolumn{2}{|c|}{$N=2$} & \multicolumn{2}{|c|}{$N=3$} & \multicolumn{2}{|c|}{$N=4$} & \multicolumn{2}{|c|}{$N=5$} \\
\hline
$n$ & \multicolumn{1}{|c}{$L^{2}\ $Error} & Rate & \multicolumn{1}{|c}{$L^{2}\ $Error} & Rate & \multicolumn{1}{|c}{$L^{2}\ $Error} & Rate & \multicolumn{1}{|c}{$L^{2}\ $Error} & Rate \\
\hline
$16$ & \multicolumn{1}{|c}{$1.65\cdot10^{-1}$} & $-$ & \multicolumn{1}{|c}{$6.28\cdot10^{-3}$} & $-$ & \multicolumn{1}{|c}{$6.28\cdot10^{-3}$} & $-$ & \multicolumn{1}{|c}{$1.81\cdot10^{-3}$} & $-$ \\
\hline
$32$ & \multicolumn{1}{|c}{$4.16\cdot10^{-2}$} & $1.99$ & \multicolumn{1}{|c}{$4.36\cdot10^{-4}$} & $3.85$ & \multicolumn{1}{|c}{$4.36\cdot10^{-4}$} & $3.85$ & \multicolumn{1}{|c}{$7.02\cdot10^{-5}$} & $4.69$ \\
\hline
$64$ & \multicolumn{1}{|c}{$1.05\cdot10^{-2}$} & $1.99$ & \multicolumn{1}{|c}{$2.53\cdot10^{-5}$} & $4.11$ & \multicolumn{1}{|c}{$2.53\cdot10^{5}$} & $4.11$ & \multicolumn{1}{|c}{$1.13\cdot10^{-6}$} & $5.95$ \\
\hline
$128$ & \multicolumn{1}{|c}{$2.62\cdot10^{-3}$} & $2.00$ & \multicolumn{1}{|c}{$1.37\cdot10^{-6}$} & $4.20$ & \multicolumn{1}{|c}{$1.37\cdot10^{-6}$} & $4.20$ & \multicolumn{1}{|c}{$1.06\cdot10^{-8}$} & $6.74$ \\
\hline
$256$ & \multicolumn{1}{|c}{$6.55\cdot10^{-4}$} & $2.00$ & \multicolumn{1}{|c}{$8.08\cdot10^{-8}$} & $4.09$ & \multicolumn{1}{|c}{$8.08\cdot10^{-8}$} & $4.09$ & \multicolumn{1}{|c}{$8.35\cdot10^{-11}$} & $6.99$ \\
\hline
$512$ & \multicolumn{1}{|c}{$1.64\cdot10^{-4}$} & $2.00$ & \multicolumn{1}{|c}{$4.97\cdot10^{-9}$} & $4.02$ & \multicolumn{1}{|c}{$4.97\cdot10^{-9}$} & $4.02$ & \multicolumn{1}{|c}{$8.54\cdot10^{-13}$} & $6.61$ \\
\hline
\end{tabulary}}
\caption{Errors and estimated orders of convergence of RECAV-FD for degrees $N.$}
\label{tab:DECAVConvergence}
\end{table}

\begin{table}[h!]
\centering
\resizebox{\textwidth}{!}{
\begin{tabulary}{\textwidth}{{|C|C|C|C|C|C|C|C|C|}}
\hline
 & \multicolumn{2}{|c|}{$N=2$} & \multicolumn{2}{|c|}{$N=3$} & \multicolumn{2}{|c|}{$N=4$} & \multicolumn{2}{|c|}{$N=5$} \\
\hline
$n$ & \multicolumn{1}{|c}{$L^{2}\ $Error} & Rate & \multicolumn{1}{|c}{$L^{2}\ $Error} & Rate & \multicolumn{1}{|c}{$L^{2}\ $Error} & Rate & \multicolumn{1}{|c}{$L^{2}\ $Error} & Rate \\
\hline
$16$ & \multicolumn{1}{|c}{$1.65\cdot10^{-1}$} & $-$ & \multicolumn{1}{|c}{$7.05\cdot10^{-3}$} & $-$ & \multicolumn{1}{|c}{$7.05\cdot10^{-3}$} & $-$ & \multicolumn{1}{|c}{$2.74\cdot10^{-3}$} & $-$ \\
\hline
$32$ & \multicolumn{1}{|c}{$4.16\cdot10^{-2}$} & $1.99$ & \multicolumn{1}{|c}{$5.42\cdot10^{-4}$} & $3.70$ & \multicolumn{1}{|c}{$5.42\cdot10^{-4}$} & $3.70$ & \multicolumn{1}{|c}{$1.17\cdot10^{-4}$} & $4.55$ \\
\hline
$64$ & \multicolumn{1}{|c}{$1.04\cdot10^{-2}$} & $1.99$ & \multicolumn{1}{|c}{$3.07\cdot10^{-5}$} & $4.14$ & \multicolumn{1}{|c}{$3.07\cdot10^{-5}$} & $4.14$ & \multicolumn{1}{|c}{$1.82\cdot10^{-6}$} & $6.01$ \\
\hline
$128$ & \multicolumn{1}{|c}{$2.62\cdot10^{-3}$} & $2.00$ & \multicolumn{1}{|c}{$1.53\cdot10^{-6}$} & $4.33$ & \multicolumn{1}{|c}{$1.53\cdot10^{-6}$} & $4.33$ & \multicolumn{1}{|c}{$1.73\cdot10^{-8}$} & $6.72$ \\
\hline
$256$ & \multicolumn{1}{|c}{$6.55\cdot10^{-4}$} & $2.00$ & \multicolumn{1}{|c}{$8.42\cdot10^{-8}$} & $4.18$ & \multicolumn{1}{|c}{$8.42\cdot10^{-8}$} & $4.18$ & \multicolumn{1}{|c}{$1.44\cdot10^{-10}$} & $6.91$ \\
\hline
$512$ & \multicolumn{1}{|c}{$1.64\cdot10^{-4}$} & $2.00$ & \multicolumn{1}{|c}{$5.03\cdot10^{-9}$} & $4.07$ & \multicolumn{1}{|c}{$5.03\cdot10^{-9}$} & $4.07$ & \multicolumn{1}{|c}{$1.41\cdot10^{-12}$} & $6.67$ \\
\hline
\end{tabulary}}
\caption{Errors and estimated orders of convergence of RKL-HLLC-FD for degrees$\ N.$}
\label{tab:DKLConvergence}
\end{table}

Now, we examine the behavior of the relaxed schemes in the Sod shocktube problem. In Figure \eqref{fig:DShuOsher}, we plot the solution to the Sod shocktube problem using each relaxed scheme. We observe that the relaxed versions of each scheme tend to generate more oscillatory solutions, compared against their unrelaxed versions, whose solutions are shown in Figure \eqref{fig:ShuOsher}. However, we observe that for $n=500$, each relaxed scheme does a better job approximating the highly oscillatory region to the right. 

\begin{figure}%
    \centering
    \subfloat[\centering $n=500\ $nodes]{{\includegraphics[width=0.333\textwidth]{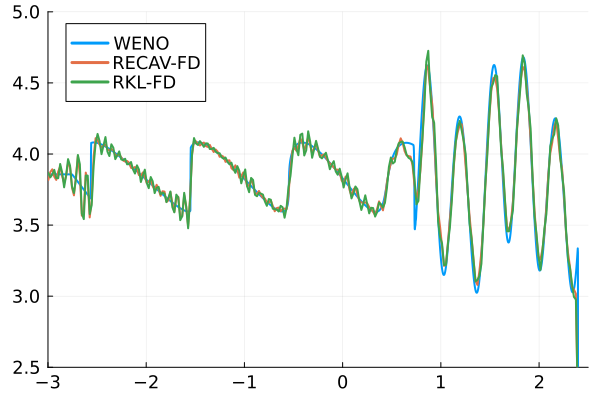} }}%
    \subfloat[\centering $n=1500\ $nodes$,\ $RECAV-FD]{{\includegraphics[width=0.333\textwidth]{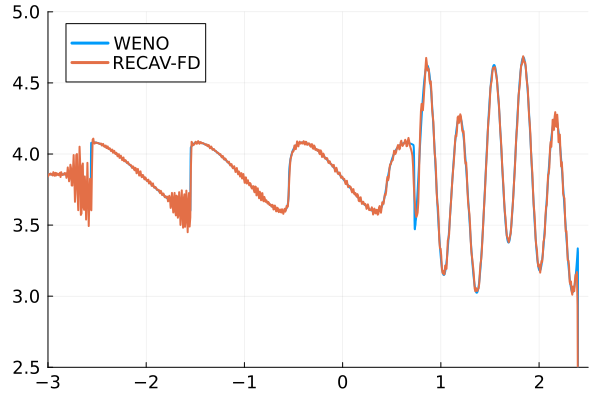} }}%
    \subfloat[\centering $n=1500$ nodes, RKL-FD-HLLC]{{\includegraphics[width=0.333\textwidth]{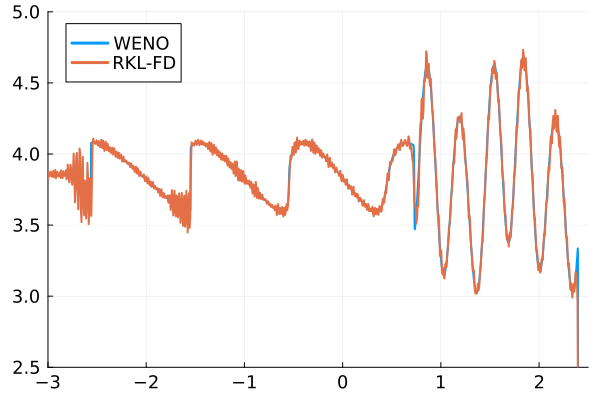} }}%
    \caption{Density of the Shu-Osher shock tube solution at final time$\ T=1.8\ $with$\ N=4.$}%
    \label{fig:DShuOsher}%
\end{figure}

In Figure \eqref{fig:ShuOsher6}, we show the same plots, but this time with a 6th order approximation. We observe that the solutions generated by the original, unrelaxed schemes are more dissipative than the solutions in Figure \eqref{fig:ShuOsher}. We also again observe that for $n=500$, each relaxed scheme does a better job approximating the highly oscillatory region to the right. However, the solutions generated by the relaxed schemes RECAV-FD and RKL-FD appear more oscillatory. We observe this behavior in general. 

\begin{figure}%
    \centering
    \subfloat[\centering $n=500\ $nodes]{{\includegraphics[width=0.333\textwidth]{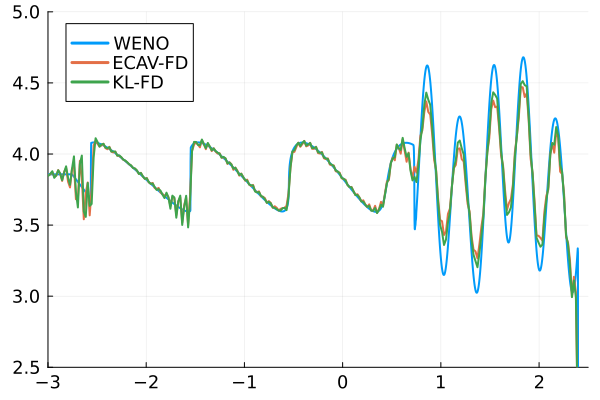} }}%
    \subfloat[\centering $n=1500\ $nodes$,\ $ECAV-FD]{{\includegraphics[width=0.333\textwidth]{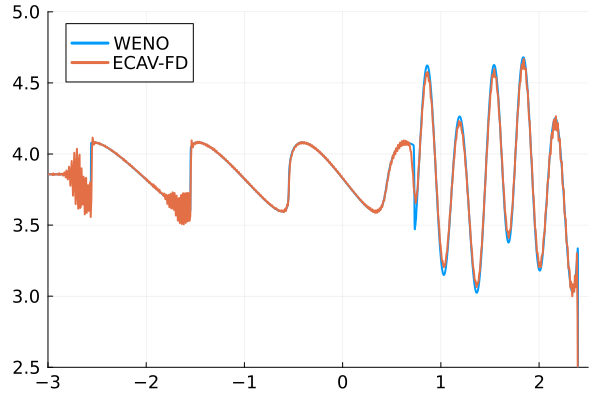} }}%
    \subfloat[\centering $n=1500\ $nodes$,\ $KL-FD]{{\includegraphics[width=0.333\textwidth]{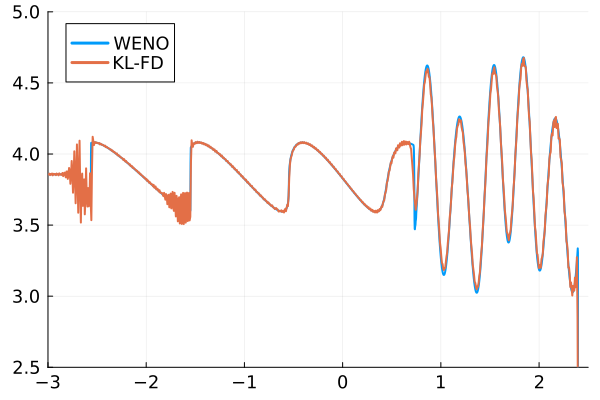} }}\\
    \subfloat[\centering $n=500\ $nodes]{{\includegraphics[width=0.333\textwidth]{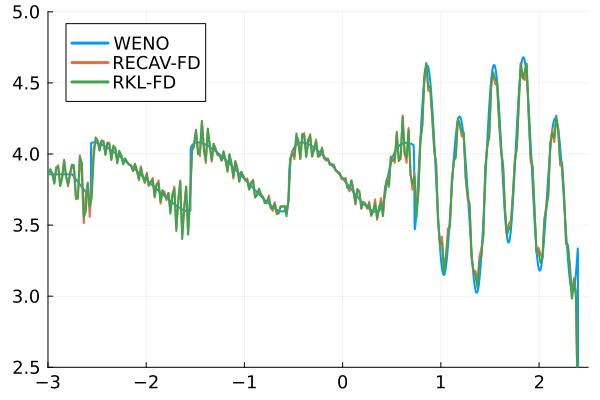} }}%
    \subfloat[\centering $n=1500\ $nodes$,\ $RECAV-FD]{{\includegraphics[width=0.333\textwidth]{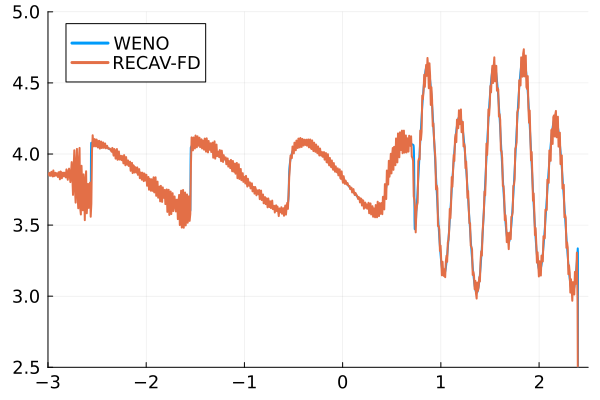} }}%
    \subfloat[\centering $n=1500\ $nodes$,\ $RKL-FD]{{\includegraphics[width=0.333\textwidth]{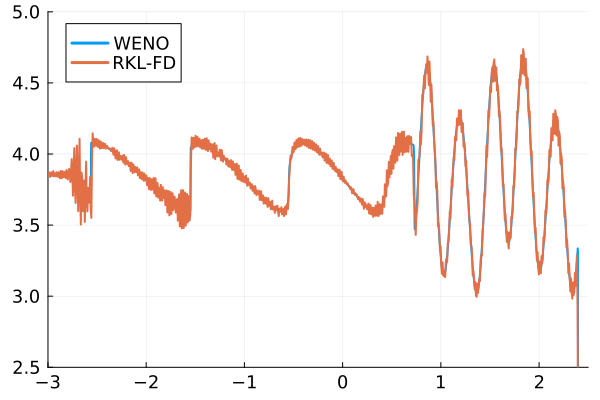} }}%
    \caption{Density of the Shu-Osher shock tube solution at final time$\ T=1.8\ $with$\ N=6.$}%
    \label{fig:ShuOsher6}%
\end{figure}

\bibliographystyle{plain}
\bibliography{references}
\end{document}